\RequirePackage{fix-cm}
\documentclass[smallextended]{svjour3}
\smartqed

    \usepackage[all]{xy}
    \usepackage{amssymb,amsmath,stmaryrd}
    \usepackage{amscd}
    \usepackage{braket,amsfonts,amsopn}
    \usepackage{bm}
    \usepackage{makecell,multirow,diagbox}
    \usepackage{mathrsfs}
    \usepackage{graphicx,epstopdf}
    \usepackage{subfigure}

\usepackage[colorlinks]{hyperref}

\DeclareMathOperator{\divg}{div}

\newcommand{\ab}[2]{\langle#1,#2\rangle}

\newcommand{\Th}{\mathcal{T}_{h}}

\newcommand{\CA}{\mathcal{A}}

\newcommand{\vertiii}[1]{{\left\vert\kern-0.25ex\left\vert\kern-0.25ex\left\vert #1
    \right\vert\kern-0.25ex\right\vert\kern-0.25ex\right\vert}}

\newcommand{\pst}{\Pi_h\widetilde{\bm{\sigma}}}
\newcommand{\ps}{\Pi_h\bm{\sigma}}
\newcommand{\tps}{\widetilde{\Pi}_h\bm{\sigma}}
\newcommand{\pu}{P_h\bm{u}}
\newcommand{\st}{\widetilde{\bm{\sigma}}}

\newcommand{\s}{\bm{\sigma}}
\newcommand{\sh}{{\bm{\sigma}_h}}
\newcommand{\eh}{\bm{e}_h}

\newcommand{\uh}{\bm{u}_h}
\newcommand{\ut}{\widetilde{\bm{u}}}
\newcommand{\bu}{\bm{u}}

\newcommand{\ma}{\mathcal{A}}
\newcommand{\bt}{\bm{\tau}}
\newcommand{\btt}{\widetilde{\bm{\tau}}}

\newcommand{\bb}{\bm{b}}
\newcommand{\putt}{P_h\widetilde{\bm{u}}}
\newcommand{\bv}{\bm{v}}
\newcommand{\bxi}{\bm{\xi}}

\newcommand{\Vh}{\bm{V}_h}

\newcommand{\bxih}{\bm{\bxi}_h}
\newcommand{\pih}{\bm{\phi}_h}

\numberwithin{theorem}{section}
\numberwithin{equation}{section}
\numberwithin{definition}{section}
\numberwithin{remark}{section}
\numberwithin{lemma}{section}

\begin{document}

\title{Superconvergent pseudostress-velocity finite element methods for the Oseen and Navier--Stokes equations}

\author{Xi Chen\and Yuwen Li
}

\institute{X. Chen \at
              Department of Mechanical Engineering, Pennsylvania State University, University Park, PA 16802, USA \\
              \email{xbc5027@psu.edu}   
           \and
           Y. Li \at
              Department of Mathematics, The Pennsylvania State University, University Park, PA 16802, USA\\ \email{yuwenli925@gmail.com}}

  \date{Received:  \  / Accepted: date}
  \maketitle
  
  \begin{abstract}
  We present a priori and superconvergence error estimates of mixed finite element methods for the pseudostress-velocity formulation of the Oseen equation. In particular, we derive superconvergence estimates for the velocity  and a priori error estimates under unstructured grids, and obtain superconvergence results for the pseudostress under certain structured grids. A variety of numerical experiments validate the theoretical results and illustrate the effectiveness of the superconvergent recovery-based adaptive mesh refinement. It is also numerically shown that the proposed postprocessing yields apparent superconvergence in a benchmark problem for the incompressible Navier--Stokes equation.
\markboth{X.~Chen, Y.~Li}{Superconvergent pseudostress-velocity mixed methods for the Oseen and NS equations}
\end{abstract}

\section{Introduction}
In fluid mechanics, the Oseen equations describe the flow of a viscous and incompressible fluid at small Reynolds numbers. Instead of dropping the advective term completely, the Oseen approximation linearizes the advective acceleration term $\bm{u}\cdot\nabla\bm{u}$ to $\bm{b}\cdot\nabla\bm{u}$, where $\bm{b}$ represents the velocity at large distance. As a result it provides a lowest-order solution that is uniformly valid everywhere in the flow field. Since the Oseen equations partly account for the inertia terms (at large distance), they have better approximation in the far field while keeping the same order of accuracy as Stokes approximation near the body, see, e.g., \cite{kundu2015fluid}. Mathematically speaking, the Oseen equation can be viewed as a linearized Navier--Stokes equation arising from fixed point iteration. 

Let $\Omega\subset\mathbb{R}^d$ be a bounded Lipschiz domain with $d\in\{2,3\}$,  $\bm{b}, \bm{f}: \Omega\rightarrow\mathbb{R}^d$, $c: \Omega\rightarrow\mathbb{R}$ and $\bm{g}: \partial\Omega\rightarrow\mathbb{R}^d$ be given data. Let $\bm{u}: \Omega\rightarrow\mathbb{R}^d$ denote the velocity field and $p:\Omega\rightarrow\mathbb{R}$ denote the pressure subject to the zero-mean constraint
\begin{equation}\label{pmean}
\int_\Omega pdx=0.    
\end{equation}
The Oseen equation under the Dirichlet boundary condition is
\begin{subequations}\label{Oseen1}
    \begin{align}
    -\nu\Delta\bm{u}+\bm{b}\cdot\nabla\bm{u}+c\bm{u}+\nabla p&=\bm{f}\text{ in }\Omega,\label{oseeneqa}\\
    \nabla\cdot\bm{u}&=0\text{ in }\Omega,\label{divu}\\
    \bm{u}&=\bm{g}\text{ on }\partial\Omega,
\end{align}
\end{subequations}
where $\nu>0$ is the viscosity constant.
As a default, vectors such as $\bm{b}, \bm{f}, \bm{u}, \bm{g}$ are always arranged as columns unless confusion arises. In \eqref{oseeneqa},  $\nabla\bm{u}$ is the Jacobian matrix of $\bm{u},$ and we adopt the convention $\bm{b}\cdot\nabla\bm{u}:=(\nabla\bm{u})\bm{b}$. 
When $\bm{b}=\bm{0}$ (resp.~$\bm{b}=\bm{0}, c=0$), \eqref{Oseen1} reduces to the Brinkman equation (resp.~Stokes equation). Replacing $\bm{b}$ with $\bm{u}$ in \eqref{Oseen1} yields the Navier--Stokes equation 
\begin{subequations}\label{NSE1}
    \begin{align}
    -\nu\Delta\bm{u}+\bm{u}\cdot\nabla\bm{u}+\nabla p&=\bm{f}\text{ in }\Omega,\\
    \nabla\cdot\bm{u}&=0\text{ in }\Omega,\\
    \bm{u}&=\bm{g}\text{ on }\partial\Omega.
\end{align}
\end{subequations}
Numerical analysis of Oseen, Brinkman, Stokes, or Navier--Stokes equations based on the velocity-pressure formulation is extensive, see, e.g., \cite{GiraultRaviart1986,BoffiBrezziFortin2013,GuzmanNeilan2014,FalkNeilan2013} for conforming mixed methods,   \cite{ccs2006,CCNP2013,CockburnSayas2014,CesmeliogluCockburnQiu2017,FuQiu2019} for hybridized discontinuous Galerkin methods, \cite{Mulin2014,WangYe2016,LLC2016} for weak Galerkin methods, and \cite{ABMV2014} for virtual element methods.

Let $\bm{I}\in\mathbb{R}^{d\times d}$ denote the identity matrix. The nonsymmetric pseudostress and symmetric stress are
\begin{align*}
    \bm{\sigma}&:=\nu\nabla\bm{u}-p\bm{I},\\
\bm{\sigma}^S&:=\frac{\nu}{2}(\nabla\bm{u}+(\nabla\bm{u})^\top)-p\bm{I},
\end{align*}
respectively.
In a series of papers \cite{Caiwang2007,CaiWangZhang2010a,CaiTongVassilevskiWang2010b}, Cai et al.~proposed the pseudostress-velocity formulation of the Stokes and Navier--Stokes equations, which could be discretized by classical Raviart--Thomas or Brezzi--Douglas--Marini elements without any stabilization term, see \eqref{mixed}. Let $\text{Tr}\bm{\sigma}$ be the trace of the matrix $\bm{\sigma}$ and
$$\mathcal{A}\bm{\sigma}:=\bm{\sigma}-\frac{1}{d}(\text{Tr}\bm{\sigma})\bm{I}$$
denote the deviatoric part of $\bm{\sigma}$. 
Direct calculation shows that the Oseen equation \eqref{Oseen1} equation is equivalent to the following pseudostress-velocity formulation \begin{equation}\label{Oseen}
    \begin{aligned}
  \mathcal{A}\bm{\sigma}&=\nu\nabla\bm{u},\\ -\nabla\cdot\bm{\sigma}+\nu^{-1}\bm{b}\cdot\mathcal{A}\bm{\sigma}+c\bm{u}&=\bm{f}
    \end{aligned}
\end{equation}
subject to the constraints $$\bm{u}|_{\partial\Omega}=\bm{g},\quad\int_\Omega\text{Tr}\bm{\sigma}dx=0.$$ 
Here $\nabla\cdot\bm{\sigma}$ is the row-wise divergence of the matrix-valued function $\bm{\sigma}$.

In contrast to extensive numerical results on velocity-pressure formulation of Stokes-related equations, numerical analysis of pseudostress-based methods is restricted to classical mixed methods \cite{CaiWangZhang2010a,Gatica2014}, virtual element methods \cite{Caceres2017,Caceres2017b},  discontinuous Galerkin methods \cite{Qian2019}, and adaptive mixed methods \cite{CKP2011,CGS2013,HuYu2018,Li2020MCOM} for the Stokes or Brinkman equation. There are a few works  devoted to the mixed formulation of the Oseen equation \eqref{Oseen}, see, e.g., \cite{Park2014} for a lowest order upwinded mixed method on rectangular meshes and \cite{BarriosCasconGonzalez2017,BarriosCasconGonzalez2020} for least-squares mixed methods. It is the purpose of this paper to shed light on a priori and superconvergence analysis of pseudostress-velocity mixed methods for the Oseen and Navier--Stokes equations.

In this paper, we shall develop a priori and superconvergence error estimates for \eqref{Oseen}. To the best of our knowledge, even a priori error estimates of \eqref{Oseen} is not available in literature. We emphasize that the presence of lower order terms is a major challenge in error analysis of mixed methods, see, e.g.,   \cite{DouglasRoberts1985,Demlow2002} for elliptic equations with lower order terms and \cite{ArnoldLi2017} for the perturbed Hodge--Laplace equation. For instance, due to the convection term $\bm{b}\cdot\mathcal{A}\bm{\sigma}$ and variable coefficient $c$, the construction of the discrete inf-sup condition is far from obvious. 
Hence we shall use the duality argument to derive a priori error estimates, which also yields improved decoupled error estimates.   It is noted that mixed methods for the scalar elliptic equation
    \begin{equation}\label{elliptic}
    -\nabla\cdot(\bm{A}\nabla u+\bm{b}u)+cu=f
    \end{equation}
    have been analyzed in \cite{DouglasRoberts1985,Demlow2002}, where $\bm{A}$ is a uniformly elliptic matrix-valued coefficient. In contrast, 
    the deviatoric operator $\CA$ in \eqref{Oseen1} is singular, which is a major difficulty in our analysis. 

    Furthermore, we shall derive  superconvergence results for $\bm{\sigma}$ based  on  postprocessing and a new estimate on $\|\Pi_h\bm{\sigma}-\bm{\sigma}_h\|$ for several families of RT elements. Even for the Stokes equation, such estimates are not known in literature. 
The superconvergence postprocessing procedure for $\bm{\sigma}_h$ and $\bm{u}_h$ is element-wise and flexible,  see Section \ref{secpost} and see, e.g.,  \cite{ZZ1992a,CKP2011,BankXu2003a,ZhangNaga2005,HuangLiWu2013,BankLi2019,Li2021JSC} for  related postprocessing schemes in the literature.  Since the postprocessing scheme is independent of the underlying physical model, we also test the proposed  postprocessing in the benchmark problem for steady incompressible Navier--Stokes equation, where apparent superconvergence phenomena in $\bm{\sigma}$ and $\bm{u}$ are observed.
As for superconvergence of Stokes-type equations in velocity-pressure form, readers are referred to e.g.,   \cite{WangYe2001,Ye2001,CuiYe2009} for superconvergence by two-grid $L^2$-projections,  \cite{Pan1997,Eichel2011,LN2008} for superconvergent recovery of lowest-order methods, and \cite{CockburnSayas2014,FuQiu2019} for superconvergent hybridized discontinuous Galerkin methods.

\subsection{Preliminary notation}
Given a vector space $X$, let $[X]^n$ denote the Cartesian product of $n$ copies of $X$ and $[X]^{n\times n}$ the space of $n\times n$ matrices whose components are contained in $X$. Let $H(\divg,\Omega)=\{\bm{v}\in [L^2(\Omega)]^d: \nabla\cdot\bm{v}\in L^2(\Omega)\}$ be the usual $H(\divg)$ space. Let 
\begin{align*}
\bm{V}&=[L^2(\Omega)]^d,\\
    \bm{\Sigma}&=\{\bm{\tau}\in[L^2(\Omega)]^{d\times d}: \bm{\tau}_i\in H(\divg,\Omega),~1\leq i\leq d,~\int_\Omega\text{Tr}\bm{\tau}dx=0\},
\end{align*}
where $\bm{\tau}_i$ is the $i$-th row of $\bm{\tau}.$  For scalar-, vector-, or matrix-valued functions, we use $(\cdot,\cdot)$ and $\langle\cdot,\cdot\rangle$ to denote the $L^2(\Omega)$- and $L^2(\partial\Omega)$-inner product, respectively.
The variational formulation of \eqref{Oseen} seeks $\bm{\sigma}\in\bm{\Sigma}$ and $\bm{u}\in\bm{V}$ such that
\begin{equation}\label{cts}
    \begin{aligned}
           (\mathcal{A}\bm{\sigma},\bm{\tau})+\nu(\nabla\cdot\bm{\tau},\bm{u})&=\nu\ab{\bm{g}}{\bm{\tau}\bm{n}},\quad\bm{\tau}\in\bm{\Sigma},\\ -\nu(\nabla\cdot\bm{\sigma},\bm{v})+(\bm{b}\cdot\mathcal{A}\bm{\sigma},\bm{v})+\nu(c\bm{u},\bm{v})&=\nu(\bm{f},\bm{v}),\quad\bm{v}\in\bm{V},
    \end{aligned}
\end{equation}
where $\bm{n}$ is the outward normal to $\partial\Omega.$

Let $\Th$ be a conforming simplicial or cubical partition of $\Omega$. Given an element $K\in\Th,$ $r_K$ and $\rho_K$ denote radii of circumscribed and inscribed spheres of $K$, $h_K$ is the diameter of $K$, and $h:=\max_{K\in\Th}h_K<1$ is the mesh size of $\Th.$ We assume that the aspect ratio of elements in $\Th$ is uniformly bounded, i.e.,
$$\max_{K\in\Th}r_K/\rho_K:={C}_{\text{mesh}}<\infty$$
for some absolute constant $C_{\text{mesh}}$.
Given an integer $r\geq0,$ let $\widetilde{\bm{\Sigma}}_r(K)$, $V_r(K)$ be suitable finite-dimensional vector spaces defined on $K$ and 
\begin{align*}
    \widetilde{\bm{\Sigma}}_h&:=\{\bm{\tau}\in H(\divg,\Omega): \bm{\tau}|_K\in\widetilde{\bm{\Sigma}}_r(K)~\text{for}~K\in\Th\},\\
    V_h&:=\{v\in L^2(\Omega): v|_K\in V_r(K)~\text{for}~K\in\Th\}.
\end{align*}
For instance, when $\Th$ is a simplicial partition, one could take 
\begin{equation}\label{RT}
    \widetilde{\bm{\Sigma}}_r(K):=[\mathcal{P}_{r}(K)]^d+\mathcal{P}_{r}(K)\bm{x},\quad V_r(K):=\mathcal{P}_r(K),
\end{equation}
where $\mathcal{P}_{r}(K)$ is the space of polynomials of degree at most $r$ on $K$, and $\bm{x}=(x_1,\ldots,x_d)^\top$ denotes the position vector. In this case, the corresponding $\widetilde{\bm{\Sigma}}_h\times V_h$ is the classical Raviart--Thomas (RT) \cite{RaviartThomas1977} space. Another possible choice is
\begin{equation}\label{BDM}
    \widetilde{\bm{\Sigma}}_r(K):=[\mathcal{P}_{r+1}(K)]^d,\quad V_r(K):=\mathcal{P}_r(K),
\end{equation}
which corresponds to the Brezzi--Douglas--Marini (BDM) \cite{BrezziDouglasMarini1985} element.

Given a hypercube $K\subset\mathbb{R}^d$, let $\mathcal{Q}_r(K)$ denote the space of polynomials on $K$ of degree at most $r$ in $x_i$ with $1\leq i\leq d$.  For a cubical mesh $\Th$, the pair
$\widetilde{\bm{\Sigma}}_h\times V_h$ can be the cubical RT element using the shape functions
\begin{equation}\label{rectRT}
    \widetilde{\bm{\Sigma}}_r(K):=[\mathcal{Q}_r(K)]^d+\bm{x}\mathcal{Q}_r(K),\quad V_r(K):=\mathcal{Q}_r(K).
\end{equation}

Let $\widetilde{\Pi}_h: [H^1(\Omega)]^d\rightarrow\widetilde{\bm{\Sigma}}_h$ be the canonical interpolation onto $\widetilde{\bm{\Sigma}}_h$ determined by the degrees of freedom of $\widetilde{\bm{\Sigma}}_h$. Let $P_h: L^2(\Omega)\rightarrow V_h$ be the $L^2$-projection onto $V_h,$ i.e.,
\begin{equation}\label{defP}
(P_hv,w)=(v,w),\quad v\in L^2(\Omega),~w\in V_h.    
\end{equation}
Throughout the rest of this paper, we assume the commutativity property
\begin{equation}\label{com1}
    \nabla\cdot\widetilde{\Pi}_h=P_h\nabla\cdot,
\end{equation}
which is satisfied by at least the aforementioned RT, BDM, and cubical RT elements. Readers are referred to e.g., \cite{BrezziDouglasMarini1985,BDDF1987,BDFM1987} for other cubical or rectangular mixed elements that satisfy \eqref{com1}.
Let $\bm{\Sigma}_h$ be the matrix version of $\widetilde{\bm{\Sigma}}_h$, i.e.,
\begin{align*}
    \bm{\Sigma}_h:=\{\bm{\tau}\in\bm{\Sigma}: \bm{\tau}_i\in \widetilde{\bm{\Sigma}}_h\text{ for }1\leq i\leq d\}.
\end{align*}
Let $\bm{V}_h$ be the vector-valued broken piecewise polynomial space
\begin{align*}
    \bm{V}_h:=\{\bm{v}=(v_1,v_2,\ldots,v_d)^\top\in\bm{V}: {v}_i\in{V}_h\text{ for }1\leq i\leq d\}.
\end{align*}
The mixed method for \eqref{Oseen} seeks
$(\bm{\sigma}_h,\bm{u}_h)\in\bm{\Sigma}_h\times\bm{V}_h$ satisfying
\begin{equation}\label{mixed}
    \begin{aligned} 
    (\mathcal{A}\sh,\bm{\tau})+\nu(\nabla\cdot\bm{\tau},\uh)&=\nu\ab{\bm{g}}{\bm{\tau}\bm{n}},\quad\bm{\tau}\in\bm{\Sigma}_h,\\
    -\nu(\nabla\cdot\sh,\bm{v})+(\bm{b}\cdot\mathcal{A}\sh,\bm{v})+\nu(c\uh,\bm{v})&=\nu(\bm{f},\bm{v}),\quad\bm{v}\in\bm{V}_h.
    \end{aligned}
\end{equation}

The outline of this work is as follows. In Section  \ref{secapriori}, we develop a apriori error estimates of the mixed method \eqref{mixed} and supercloseness estimate on $\|P_h\bm{u}-\bm{u}_h\|$. In Section \ref{secsuper}, we develop supercloseness estimate on $\|\Pi_h\bm{\sigma}-\bm{\sigma}_h\|$. Section \ref{secpost} is devoted to superconvergent  postprocessing procedures. Numerical results for smooth, singular, convection-dominated, and nonlinear problems are presented in Section \ref{secexp}. Throughout the rest of this paper,
we say $A\lesssim B$ provided $A\leq CB,$ 
where $C>0$ is a generic constant dependent solely on $\bm{b}$, $c$, $\Omega$, $C_{\text{mesh}}$, $\nu$. In the error analysis, we assume $\nu=1$ without loss of generality.

\section{A priori error estimates}\label{secapriori}
Let $\|\cdot\|_r$ denote the $\|\cdot\|_{H^r(\Omega)}$-norm, $|\cdot|_r$ be the $|\cdot|_{H^r(\Omega)}$-semi norm,  $\|\cdot\|:=\|\cdot\|_0$ be the $\|\cdot\|_{L^2(\Omega)}$-norm, and $\vertiii{\cdot}$ be the  $H(\divg)$-norm
$$\vertiii{\bm{\tau}}^2:=\|\bm{\tau}\|^2+\|\nabla\cdot\bm{\tau}\|^2.$$
In this section, we derive a priori error estimates of the mixed method \eqref{mixed} in Theorem \ref{apriori}. The same analysis implies that \eqref{mixed} admits a unique solution provided $h$ is sufficiently small.

The operator $\mathcal{A}$ is singular and satisfies
\begin{equation}\label{adjA}
(\mathcal{A}\bm{\tau}_1,\bm{\tau}_2)=(\bm{\tau}_1,\mathcal{A}\bm{\tau}_2)=(\mathcal{A}\bm{\tau}_1,\mathcal{A}\bm{\tau}_2),\quad\bm{\tau}_1, \bm{\tau}_2\in[L^2(\Omega)]^{d\times d}.
\end{equation}
In addition, it holds that (see \cite{Caiwang2007,ArnoldDouglasGupta1984}) 
\begin{equation}\label{rb}
    \|\bm{\tau}\|\lesssim\|\bm{\tau}\|_\CA+\|\nabla\cdot\bm{\tau}\|_{-1},\quad\bm{\tau}\in\bm{\Sigma},
\end{equation}
where $\|\bm{\tau}\|_\CA:=(\ma\bm{\tau},\bm{\tau})^\frac{1}{2}.$ The inequality \eqref{rb} is a key ingredient in the analysis of pseudostress-based methods. 
Define $\Pi_h: [H^1(\Omega)]^{d\times d}\rightarrow\bm{\Sigma}_h$ by
\begin{equation}
    \Pi_h\bt:=\widetilde{\Pi}_h\bt-\frac{1}{d|\Omega|}\left(\int_{\Omega}\text{Tr}\widetilde{\Pi}_h\bm{\tau}dx\right)\bm{I},
\end{equation}
where $\widetilde{\Pi}_h$ is applied to each row of $\bm{\tau}.$ By abuse of notation, we may also use $P_h: [L^2(\Omega)]^d\rightarrow\Vh$ to denote the component-wise $L^2$-projection. It follows from \eqref{com1} that
\begin{equation}\label{com}
    \nabla\cdot\Pi_h\bm{\tau}=P_h\nabla\cdot\bm{\tau},\quad\forall\bm{\tau}\in [H^1(\Omega)]^{d\times d}.
\end{equation}
For convenience, we introduce the interpolation errors
$$\bm{\rho}_h:=\bm{u}-P_h\bm{u},\quad\bm{\zeta}_h:=\bm{\sigma}-\Pi_h\bm{\sigma},$$
which can be easily estimated by (see \cite{CaiTongVassilevskiWang2010b})
\begin{subequations}\label{approx}
    \begin{align}
    \|\bm{v}-P_h\bm{v}\|&\lesssim h^s|\bm{v}|_s,\\
\|\bm{\tau}-\Pi_h\bm{\tau}\|&\lesssim h^s|\bm{\tau}|_s,\\
    \|\nabla\cdot(\bm{\tau}-{\Pi}_h\bm{\tau})\|&\lesssim h^s|\divg\bm{\tau}|_s,
    \end{align}
\end{subequations}
where $0\leq s\leq r+1$, $\int_\Omega\text{Tr}\bm{\tau}dx=0$, and $\bm{v}, \bm{\tau}$ satisfy the regularity indicated by the right hand sides. For the BDM element \eqref{BDM}, it additionally holds that
\begin{equation}\label{approxBDM}
\|\bm{\tau}-\Pi_h\bm{\tau}\|\lesssim h^s|\bm{\tau}|_s,\quad0\leq s\leq r+2.
\end{equation}
The essential errors to be estimated are
\begin{equation*}
    \bm{e}_h:=\pu-\uh,\quad\bm{\xi}_h:=\Pi_h\bm{\sigma}-\bm{\sigma}_h.
\end{equation*}
Subtracting \eqref{mixed} from \eqref{cts} (with $\nu=1$), we obtain the error equation
\begin{subequations}\label{error}
    \begin{align}
    (\ma(\s-\sh),\bt_h)+(\nabla\cdot\bm{\tau}_h,\bu-\uh)&=0,\quad\bm{\tau}_h\in\bm{\Sigma}_h,\label{error1}\\
    -(\nabla\cdot(\bm{\sigma}-\sh),\bm{v}_h)+(\bm{b}\cdot\mathcal{A}(\bm{\sigma}-\sh),\bm{v}_h)+(c(\bm{u}-\uh),\bm{v}_h)&=0,\quad\bm{v}_h\in\bm{V}_h.\label{error2}
    \end{align}
\end{subequations}
To estimate $\|\bm{e}_h\|$, we consider the adjoint problem of \eqref{cts}: Find $\widetilde{\bm{\sigma}}\in\bm{\Sigma}$ and $\widetilde{\bm{u}}\in\bm{V}$ such that
\begin{subequations}\label{dual_problem}
    \begin{align} (\mathcal{A}\widetilde{\bm{\sigma}},\bm{\tau})+(\nabla\cdot\bm{\tau},\widetilde{\bm{u}})-(\bm{b}\cdot\mathcal{A}\bm{\tau},\widetilde{\bm{u}})&=0,\quad\bm{\tau}\in\bm{\Sigma},\label{dual1}\\ -(\nabla\cdot\widetilde{\bm{\sigma}},\bm{v})+(c\widetilde{\bm{u}},\bm{v})&=(\eh,\bm{v})\quad\bm{v}\in\bm{V}.\label{dual2}
    \end{align}
\end{subequations}
In the analysis, it is  assumed that \eqref{dual_problem} admits the elliptic regularity
\begin{equation}\label{duality_assumption}
\|\widetilde{\bm{\sigma}}\|_1+\|\widetilde{\bm{u}}\|_1\lesssim\|\eh\|.
\end{equation}
The next lemma is a supercloseness estimate which is crucial for both a priori and superconvergence error analysis. 
\begin{lemma}\label{superu} It holds that
\begin{equation*}
           \|\eh\|\lesssim h\big(\vertiii{\s-\sh}+\|\bu-\uh\|+\|\bm{\rho}_h\|\big).
\end{equation*}
\end{lemma}
\begin{proof}
Taking $\bm{v}=\eh$ in \eqref{dual2}, we obtain
\begin{equation}\label{error_eq}
    \|\eh\|^2=-(\nabla\cdot\widetilde{\bm{\sigma}},\eh)+(c\widetilde{\bm{u}},\eh).
\end{equation}
Using the definition \eqref{defP} and the property \eqref{com}, we have 
\begin{equation}\label{term1}
    \begin{aligned}
    &(\nabla\cdot\widetilde{\bm{\sigma}},\eh)=(P_h\nabla\cdot\st,\pu-\bm{u}+\bm{u}-\uh)\\
    &\quad=(P_h\nabla\cdot\st,\bm{u}-\uh)=(\nabla\cdot\pst,\bu-\uh).
    \end{aligned}
\end{equation}
Combining \eqref{error1} with $\bt_h=\pst$ implies that
\begin{equation}\label{term11}
    \begin{aligned}
    &(\nabla\cdot\pst,\bu-\uh)=-(\ma(\s-\sh),\pst)\\
    &\quad=(\ma(\s-\sh),\st-\pst)-(\ma(\s-\sh),\st).
    \end{aligned}
\end{equation}
Using \eqref{adjA} and setting $\btt=\s-\sh$ in \eqref{dual1}, we have
\begin{equation*}
    \begin{aligned}
           &(\mathcal{A}(\s-\sh),\st)=(\mathcal{A}\st,\s-\sh)\\
           &=-(\nabla\cdot(\s-\sh),\widetilde{\bm{u}})+(\bm{b}\cdot\mathcal{A}(\s-\sh),\widetilde{\bm{u}})\\
          &=-(\nabla\cdot(\s-\sh),\widetilde{\bm{u}}-P_h\widetilde{\bm{u}})-(\nabla\cdot(\s-\sh),P_h\widetilde{\bm{u}})+(\bm{b}\cdot\mathcal{A}(\s-\sh),\widetilde{\bm{u}}).
    \end{aligned}
\end{equation*}
It then follows from the above equation and  \eqref{error2} with $\bv_h=\putt$ that 
\begin{equation}\label{term12}
    \begin{aligned}
           &(\mathcal{A}(\s-\sh),\st)=-(\nabla\cdot(\s-\sh),\widetilde{\bm{u}}-P_h\widetilde{\bm{u}})+(\bm{b}\cdot\mathcal{A}(\s-\sh),\widetilde{\bm{u}})\\ &\quad-(\bm{b}\cdot\mathcal{A}(\bm{\sigma}-\sh),\putt)-(c(\bm{u}-\uh),\putt).
    \end{aligned}
\end{equation}
As a result of \eqref{term1}, \eqref{term11} and \eqref{term12}, we obtain
\begin{equation}\label{newterm1}
    \begin{aligned}
    (\nabla\cdot\widetilde{\bm{\sigma}},\eh)&=(\ma(\s-\sh),\st-\pst)+(\nabla\cdot(\s-\sh),\ut-\putt)\\
    &-(\bm{b}\cdot\mathcal{A}(\s-\sh),\widetilde{\bm{u}}-\putt)+(c(\bm{u}-\uh),\putt).
    \end{aligned}
\end{equation}
On the other hand, the second term on the right hand side of \eqref{error_eq} is
\begin{equation}\label{term2}
    \begin{aligned}
    (c\ut,\bm{e}_h)&=(c\ut,\pu-\bu+\bu-\uh)\\
    &=(c\ut-P_h (c\ut),\pu-\bu)+(c(\bu-\uh),\ut).
    \end{aligned}
\end{equation}
Finally with \eqref{newterm1} and \eqref{term2}, the error in \eqref{error_eq} can be written as
\begin{equation}\label{error_eq_final_1}
    \begin{aligned}
    &\|\eh\|^2=(\ma(\s-\sh),\pst-\st)-(\nabla\cdot(\s-\sh),\ut-\putt)\\
    &\quad+(\bb\cdot\ma(\s-\sh),\ut-\putt)+(c(\bu-\uh),\widetilde{\bm{u}}-\putt)\\
    &\quad+(c\ut-P_h(c\ut),\pu-\bu).
    \end{aligned}
\end{equation}
Using \eqref{error_eq_final_1}, \eqref{approx}, and the Cauchy--Schwarz inequality, we obtain
\begin{equation*}
    \|\eh\|^2\lesssim h\|\st\|_1\|\s-\sh\|+h\|\widetilde{\bm{u}}\|_1\big(\vertiii{\s-\sh}+\|\bm{u}-\bm{u}_h\|+\|\bm{\rho}_h\|\big).
\end{equation*}
Combining the above estimate  
with \eqref{duality_assumption} completes the proof.
\qed\end{proof}

Lemma \ref{superu} is a supercloseness estimate, i.e., $P_h\bm{u}$ and $\bm{u}_h$ are much closer than the distance predicted by standard a priori error estimates. However, a priori error estimates on $\|\bm{\sigma}-\bm{\sigma}_h\|$ and $\|\nabla\cdot(\bm{\sigma}-\bm{\sigma}_h)\|$ are not known at the moment. When deriving a priori error estimates, we need the $L^2$ and negative norm estimates of $\nabla\cdot\bm{\xi}_h.$
\begin{lemma}\label{divneg}
It holds that
\begin{subequations}
\begin{align}
\|\nabla\cdot \bxi_h\|&\lesssim\|\bm{\zeta}_h\|+\|\bxi_h\|+\|\eh\|+h\|\bm{\rho}_h\|,\label{divxil2}\\
\|\nabla\cdot\bm{\xi}_h\|_{-1}&\lesssim h\big(\vertiii{\bm{\zeta}_h}+\vertiii{\bxi_h}+\|\bm{\rho}_h\|\big)+\|\eh\|. \label{divxin1}
\end{align}
\end{subequations}
\end{lemma} 
Lemme \ref{divneg} is also essential for proving the supercloseness estimate in Theorem \ref{supersigma} and the proof is postponed in the end of this section. In this work, we say $h$ is sufficiently small provided $h\leq h_0,$ where $h_0$ is an absolute constant relying on $\Omega$, $\bm{b}$, $c$, $C_{\text{mesh}}$, $\nu$.  Now we are in a position to present the first main result in this paper. 
\begin{theorem}\label{apriori} For sufficiently small $h$, it holds that
\begin{subequations}\label{prio_result}
    \begin{align}
          \|\bu-\bu_h\|&\lesssim \|\bm{\rho}_h\|+h\vertiii{\bm{\zeta}_h},\\
        \vertiii{\s-\sh}&\lesssim \vertiii{\bm{\zeta}_h}+h\|\bm{\rho}_h\|,\\
        \|\s-\sh\|&\lesssim \|\bm{\zeta}_h\|+h\|\nabla\cdot\bm{\zeta}_h\|+h\|\bm{\rho}_h\|.\label{apriori3}
    \end{align}
\end{subequations}
\end{theorem}
\begin{proof}
Using Lemma \ref{superu} and the triangle inequality \begin{align*}
\|\bu-\uh\|&\leq\|\bm{\rho}_h\|+\|\eh\|,\\
\vertiii{\bm{\sigma}-\bm{\sigma}_h}&\leq\vertiii{\bm{\zeta}_h}+\vertiii{\bm{\xi}_h},
\end{align*}
we obtain
\begin{equation}\label{eh_error_general} 
\|\eh\|\lesssim h\big( \vertiii{\bm{\zeta}_h}+\vertiii{\bm{\xi}_h}+\|\bm{\rho}_h\|\big)+h\|\eh\|.
\end{equation}
Plugging  \eqref{divxil2} into \eqref{eh_error_general} and kicking $h\|\bm{e}_h\|$ (with sufficiently small $h$) back to the left hand side  yields
\begin{equation}\label{eh_error_general_2}
\|\eh\|\lesssim h\big(\vertiii{\bm{\zeta}_h}+\|\bxi_h\|+\|\bm{\rho}_h\|\big).
\end{equation}
On the other hand, with help of \eqref{error1} with  $\bm{\tau}_h=\bm{\xi}_h$, we obtain
\begin{equation}\label{eh_error}
\begin{aligned}
(\ma\bxi_h,\bxi_h)&=(\ma(\ps-\s),\bxi_h)-(\nabla\cdot\bxi_h,\bu-\uh)\\
    &=(\ma(\ps-\s),\bxi_h)-(\nabla\cdot\bxi_h,P_h\bu-\uh).
\end{aligned}
\end{equation}
It follows from \eqref{rb}, \eqref{eh_error},  and a Young's inequality with $\varepsilon>0$ that
\begin{equation*}
\begin{aligned}
\|\bxi_h\|^2&\lesssim(\ma\bxi_h,\bxi_h)+\|\nabla\cdot\bxi_h\|_{-1}^2\\
&\lesssim\|\bm{\zeta}_h\|\|\bxi_h\|+\|\nabla\cdot\bxi_h\|\|\eh\|+\|\nabla\cdot\bxi_h\|_{-1}^2\\
&\leq\frac{\varepsilon^2}{2}\|\bm{\xi}_h\|^2+\frac{\varepsilon^{-2}}{2}\|\bm{\zeta}_h\|^2+\frac{\varepsilon^2}{2}\|\nabla\cdot\bxi_h\|^2+\frac{\varepsilon^{-2}}{2}\|\eh\|^2+\|\nabla\cdot\bm{\xi}_h\|_{-1}^2,
\end{aligned}
\end{equation*}
or equivalently
\begin{equation}\label{xiest1}
\|\bxi_h\|\leq C^*\big({\varepsilon}\|\bm{\xi}_h\|+{\varepsilon^{-1}}\|\bm{\zeta}_h\|+{\varepsilon}\|\nabla\cdot\bm{\xi}_h\|+{\varepsilon^{-1}}\|\eh\|+\|\nabla\cdot\bm{\xi}_h\|_{-1}\big),
\end{equation}
where $C^*$ is independent of $\varepsilon$. Using  \eqref{xiest1}, Lemma \ref{divneg}, and the triangle inequality, we  deduce that
\begin{equation*}
\begin{aligned}
\|\bxi_h\|&\leq C^*\big\{\varepsilon\big(\|\bxi_h\|+\|\bm{\zeta}_h\|+\|\eh\|+h\|\bm{\rho}_h\|\big)\\
&+{\varepsilon^{-1}}\big(\|\bm{\zeta}_h\|+\|\eh\|\big)+h\vertiii{\bm{\zeta}_h}+h\|\bxi_h\|+h\|\bm{\rho}_h\|+\|\eh\|\big\}.
\end{aligned}
\end{equation*}
In the above estimate, it suffices to choose sufficiently small $\varepsilon$ and $h$ to obtain
\begin{equation}\label{xierr}
    \|\bxi_h\|\leq C \big(\|\bm{\zeta}_h\|+h\|\nabla\cdot\bm{\zeta}_h\|+\|\eh\|+h\|\bm{\rho}_h\|\big).
\end{equation}
Plugging \eqref{xierr} into \eqref{eh_error_general_2} with sufficiently small $h$ then leads to
\begin{equation}\label{eherr}
\|\eh\|\leq C h\big(\vertiii{\bm{\zeta}_h}+\|\bm{\rho}_h\|\big).
\end{equation}
Therefore we close the loop.
As a consequence of  \eqref{xierr},  \eqref{eherr}, \eqref{divxil2}, we obtain
\begin{equation}\label{xiherr}
\begin{aligned}
 \|\bxi_h\|&\leq C \big(\|\bm{\zeta}_h\|+h\|\nabla\cdot\bm{\zeta}_h\|+h\|\bm{\rho}_h\|\big),\\
 \|\nabla\cdot\bxi_h\|&\leq C \big(\|\bm{\zeta}_h\|+h\|\nabla\cdot\bm{\zeta}_h\|+h\|\bm{\rho}_h\|\big).
\end{aligned}
\end{equation}
Theorem \ref{apriori} eventually follows from \eqref{eherr}, \eqref{xiherr} and a triangle inequality.
\qed\end{proof}

For sufficiently smooth $(\bm{\sigma},\bm{u})$, Theorem \ref{apriori} with \eqref{approx} yields
\begin{subequations}\label{convrate}
    \begin{align}
          \|\bu-\bu_h\|&\lesssim h^{r+1}\big( |\bm{u}|_{r+1}+h|\bm{\sigma}|_{r+1}+h|\nabla\cdot\bm{\sigma}|_{r+1}\big),\\
        \vertiii{\s-\sh}&\lesssim  h^{r+1}\big(|\bm{\sigma}|_{r+1}+|\nabla\cdot\bm{\sigma}|_{r+1}+h|\bm{u}|_{r+1}\big).\label{convrateb}
    \end{align}
\end{subequations}
When $\bm{\Sigma}_h$ is based on the BDM element \eqref{BDM}, the following improved error estimate follows from  \eqref{apriori3} and \eqref{approxBDM}.
\begin{equation}\label{convrateBDM}
    \|\bm{\sigma}-\bm{\sigma}_h\|\lesssim h^{r+2}\big(|\bm{\sigma}|_{r+2}+|\bm{u}|_{r+1}\big).
\end{equation}

The well-posedness \eqref{mixed} follows from the same analysis in the proof of Theorem \ref{apriori}.
\begin{corollary}
For sufficiently small $h$, the mixed method \eqref{mixed} has a unique solution.
\end{corollary}
\begin{proof}
To establish the existence and uniqueness of the solution to \eqref{mixed}, it suffices to show the uniqueness because of linearity. Suppose $(\bm{\sigma}_h,\bm{u}_h)$ and $(\widehat{\bm{\sigma}}_h,\widehat{\bm{u}}_h)$ are both solutions to \eqref{mixed}, then we have the error equation
\begin{equation}\label{errorex}
    \begin{aligned}
    (\ma\bm{e}_\sigma,\bt_h)+(\nabla\cdot\bm{\tau}_h,\bm{e}_u)&=0,\quad\bm{\tau}_h\in\bm{\Sigma}_h,\\
    -(\nabla\cdot \bm{e}_\sigma,\bm{v}_h)+(\bm{b}\cdot\mathcal{A}\bm{e}_\sigma,\bm{v}_h)+(c\bm{e}_u,\bm{v}_h)&=0,\quad\bm{v}_h\in\bm{V}_h,
    \end{aligned}
\end{equation}
where $\bm{e}_\sigma:=\bm{\sigma}_h-\widehat{\bm{\sigma}}_h$, $\bm{e}_u:=\bm{u}_h-\widehat{\bm{u}}_h.$ Consider the dual problem
\begin{equation}\label{dualex}
    \begin{aligned}
    (\ma\widetilde{\bm{\sigma}},\bt)+(\nabla\cdot\bm{\tau},\widetilde{\bm{u}})-(\bm{b}\cdot\mathcal{A}\bm{\tau},\widetilde{\bm{u}})&=0,\quad\bm{\tau}\in\bm{\Sigma},\\
    -(\nabla\cdot\widetilde{\bm{\sigma}},\bm{v})+(c\widetilde{\bm{u}},\bm{v})&=(\bm{e}_u,v),\quad\bm{v}\in\bm{V}.
    \end{aligned}
\end{equation}
It then follows from \eqref{errorex}, \eqref{dualex} and the same analysis in Lemma \ref{superu} that
\begin{equation}\label{eusigma}
\|\bm{e}_u\|\lesssim h\vertiii{\bm{e}_\sigma},
\end{equation}
provided $h$ is sufficiently small.
The same argument for proving Lemma \ref{divneg} yields
\begin{equation}\label{divex}
\begin{aligned}
\|\nabla\cdot\bm{e}_\sigma\|&\lesssim\|\bm{e}_\sigma\|+\|\bm{e}_u\|,\\
\|\nabla\cdot\bm{e}_\sigma\|_{-1}&\lesssim h\vertiii{\bm{e}_\sigma}+\|\bm{e}_u\|. 
\end{aligned}
\end{equation}
Following the proof of \eqref{xierr} and using \eqref{eusigma}, \eqref{divex}, we obtain
\begin{equation}\label{esigmau}
\|\bm{e}_\sigma\|\lesssim\|\bm{e}_u\|,
\end{equation}
when $h$ is small enough. Finally combining \eqref{eusigma}, \eqref{esigmau}, \eqref{divex} yields 
$$\|\bm{e}_u\|\lesssim h\|\bm{e}_u\|.$$
Hence $\bm{e}_u=\bm{0}$ and $\bm{e}_\sigma=\bm{0}$ provided $h$ is sufficiently small.
\qed\end{proof}

\begin{proof}[Proof of Lemma \ref{divneg}]
Let $\bv_h=\nabla\cdot\bm{\xi}_h/\|\nabla\cdot\bm{\xi}_h\|$. Using \eqref{com} and \eqref{error2}, we arrive at
\begin{equation}\label{divxil2est}
    \begin{aligned}
    &\|\nabla\cdot \bxi_h\|=(\nabla\cdot\bxi_h,\bm{v}_h)=(P_h\nabla\cdot\bm{\sigma}-\nabla\cdot\sh,\bm{v}_h)=(\nabla\cdot(\s-\sh),\bm{v}_h)\\
    &\quad=(\bb\cdot\ma(\s-\sh),\bm{v}_h)+(c(P_h\bu-\uh),\bm{v}_h)+((c-c_h)(\bu-P_h\bm{u}),\bm{v}_h).
    \end{aligned}
\end{equation}
where $c_h$ is the piecewise average of $c$ w.r.t.~$\Th.$ The estimate
\eqref{divxil2} then follows from \eqref{divxil2est}, the Cauchy--Schwarz and triangle inequalities.

To prove \eqref{divxin1},
let $\bm{v}\in [H_0^1(\Omega)]^d$ with $\|\bm{v}\|_1=1$. It follows from \eqref{com}, \eqref{defP}, and \eqref{error2} that
\begin{equation}\label{divtotal}
    \begin{aligned}
    &(\nabla\cdot\bm{\xi}_h,\bm{v})=(P_h\nabla\cdot\bm{\sigma}-\nabla\cdot\bm{\sigma},\bm{v})+(\nabla\cdot\bm{\sigma}-\nabla\cdot\bm{\sigma}_h,\bm{v})\\
    &\quad=(P_h\nabla\cdot\bm{\sigma}-\nabla\cdot\bm{\sigma},\bm{v}-P_h\bm{v})+(\nabla\cdot\bm{\sigma}-\nabla\cdot\bm{\sigma}_h,\bm{v}-P_h\bm{v})\\
    &\qquad+(\bm{b}\cdot\mathcal{A}(\bm{\sigma}-\bm{\sigma}_h),\bm{v})+(\bm{b}\cdot\mathcal{A}(\bm{\sigma}-\bm{\sigma}_h),P_h\bm{v}-\bm{v})\\
    &\qquad+(c(\bm{u}-P_h\bm{u}),P_h\bm{v})+(c(P_h\bm{u}-\bm{u}_h),P_h\bm{v}).
    \end{aligned}
\end{equation}
Using \eqref{adjA} and \eqref{error1}, we can rewrite  $(\bm{b}\cdot\mathcal{A}(\bm{\sigma}-\bm{\sigma}_h),\bm{v})$ as
\begin{equation}\label{divterm1}
    \begin{aligned}
    &\quad(\bm{b}\cdot\mathcal{A}(\bm{\sigma}-\bm{\sigma}_h),\bm{v})\\
    &=(\mathcal{A}(\bm{\sigma}-\bm{\sigma}_h),\bm{v}\bm{b}^\top)=(\mathcal{A}(\bm{\sigma}-\bm{\sigma}_h),\mathcal{A}(\bm{v}\bm{b}^\top))\\
    &=(\mathcal{A}(\bm{\sigma}-\bm{\sigma}_h),\Pi_h\mathcal{A}(\bm{v}\bm{b}^\top))+(\mathcal{A}(\bm{\sigma}-\bm{\sigma}_h),(\bm{I}-\Pi_h)\mathcal{A}(\bm{v}\bm{b}^\top))\\
    &=-(\nabla\cdot\Pi_h\mathcal{A}(\bm{v}\bm{b}^\top),\bm{u}-\bm{u}_h)+(\mathcal{A}(\bm{\sigma}-\bm{\sigma}_h),(\bm{I}-\Pi_h)\mathcal{A}(\bm{v}\bm{b}^\top))\\
    &=-(\nabla\cdot\Pi_h\mathcal{A}(\bm{v}\bm{b}^\top),P_h\bm{u}-\bm{u}_h)+(\mathcal{A}(\bm{\sigma}-\bm{\sigma}_h),(\bm{I}-\Pi_h)\mathcal{A}(\bm{v}\bm{b}^\top)).
    \end{aligned}
\end{equation}
On the other hand,
\begin{equation}\label{divterm2}
(c(\bm{u}-P_h\bm{u}),P_h\bm{v})=((c-c_h)(\bm{u}-P_h\bm{u}),P_h\bm{v}).
\end{equation}
Combining \eqref{divtotal} with \eqref{divterm1} and \eqref{divterm2} and using \eqref{approx}, we obtain
\begin{equation}\label{div1}
    \begin{aligned}
    &\|\nabla\cdot\bm{\xi}_h\|_{-1}=\sup_{\|\bm{v}\|_1=1}(\nabla\cdot\bm{\xi}_h,\bm{v})\\
    &\lesssim h\big(\|P_h\nabla\cdot\bm{\sigma}-\nabla\cdot\bm{\sigma}\|+\vertiii{\bm{\sigma}-\bm{\sigma}_h}+\|\bm{u}-P_h\bm{u}\|\big)\\
    &\quad+(1+\|\nabla\cdot\Pi_h\mathcal{A}(\bm{v}\bm{b}^\top)\|)\|\eh\|.
    \end{aligned}
\end{equation}
It follows from \eqref{com} and elementary calculation that 
$$\|\nabla\cdot\Pi_h\mathcal{A}(\bm{v}\bm{b}^\top)\|=\|P_h\nabla\cdot\mathcal{A}(\bm{v}\bm{b}^\top)\|\leq\|\nabla\cdot\mathcal{A}(\bm{v}\bm{b}^\top)\|\leq C\|\bm{v}\|_1\leq C.$$ 
Therefore using the previous estimate and \eqref{div1}, we obtain
\begin{equation*}
    \|\nabla\cdot\bm{\xi}\|_{-1}\lesssim h\|\nabla\cdot\Pi_h\bm{\sigma}-\nabla\cdot\bm{\sigma}\|+h\vertiii{\bm{\sigma}-\bm{\sigma}_h}+h\|\bm{\rho}_h\|+\|\eh\|.
\end{equation*}
The proof is complete.
\qed\end{proof}
\section{Superconvergence on pseudostress}\label{secsuper}
Using Lemma \ref{superu} and \eqref{convrate}, we obtain the supercloseness estimate on $\eh=P_h\bm{u}-\bm{u}_h$.
\begin{equation}\label{ratesuperu}
           \|\eh\|\lesssim h^{r+2}\big(|\bm{\sigma}|_{r+1}+|\nabla\cdot\bm{\sigma}|_{r+1}+|\bm{u}|_{r+1}\big).
\end{equation} 
In this section, we shall prove an improved estimate for $\bm{\xi}_h=\Pi_h\bm{\sigma}-\bm{\sigma}_h$. To this end,
$\bm{\Sigma}_h$ is endowed with the $\Lambda$-inner product
\begin{equation*}
    \Lambda(\bm{\tau}_1,\bm{\tau}_2):=(\ma\bm{\tau}_1,\bm{\tau}_2)+(\nabla\cdot\bm{\tau}_1,\nabla\cdot\bm{\tau}_2),\quad\bm{\tau}_1,\bm{\tau}_2\in\bm{\Sigma}_h.
\end{equation*}
Let $\bm{Z}_h:=\ker\big(\nabla\cdot|_{\bm{\Sigma}_h}\big)$ and $\bm{Z}^\perp_h$ be the orthogonal complement of $\bm{Z}_h$ w.r.t.~$\Lambda(\cdot,\cdot)$ in $\bm{\Sigma}_h$. We shall decompose $\bxih\in\bm{\Sigma}_h$ by the discrete Helmholtz decomposition
\begin{equation}\label{decomp}
     \bm{\Sigma}_h=\bm{Z}_h\oplus_\Lambda\bm{Z}_h^\perp,
\end{equation}
which can be analyzed using tools in finite element exterior calculus (FEEC), see, e.g., \cite{ArnoldFalkWinther2006,ArnoldFalkWinther2010}. In the theory of FEEC, an essential ingredient is the  bounded projections that commute with exterior derivatives. In particular, there exist {projections}  $\widetilde{\pi}_h: [L^2(\Omega)]^d\rightarrow\widetilde{\bm{\Sigma}}_h$ and ${Q}_h: L^2(\Omega)\rightarrow{V}_h$ (see, e.g., \cite{ArnoldFalkWinther2006,ChristiansenWinther2008}) such that 
\begin{equation}\label{pitilde}
\begin{aligned}
    &\widetilde{\pi}_h|_{\widetilde{\bm{\Sigma}}_h}=\text{Id},\quad{Q}_h|_{{V}_h}=\text{Id},\\
    &\nabla\cdot\widetilde{\pi}_h\bm{\tau}={Q}_h\nabla\cdot\bm{\tau}\text{ for }\bm{\tau}\in H(\text{div},\Omega),\\
    &\|\widetilde{\pi}_h\bm{\tau}\|\lesssim\|\bm{\tau}\|\text{ for }\bm{\tau}\in [L^2(\Omega)]^d,
\end{aligned}
\end{equation}
where $\text{Id}$ is the identity mapping. Starting from $\widetilde{\pi}_h, Q_h,$ one can easily obtain commuting bounded projections onto $\bm{\Sigma}_h$ and $\bm{V}_h$. Let
\begin{equation*}
    \pi_h\bt:=\widetilde{\pi}_h\bt-\frac{1}{d|\Omega|}\left(\int_{\Omega}\text{Tr}(\widetilde{\pi}_h\bm{\tau})dx\right)\bm{I},
\end{equation*}
where $\widetilde{\pi}_h$ is applied to each row of $\bm{\tau}$. Similarly, $Q_h$ can applied to each component of a vector-valued function in $[L^2(\Omega)]^d$.
Using the property of $\widetilde{\pi}_h$ in \eqref{pitilde}, we have 
\begin{subequations}
\begin{align}
    \pi_h|_{\bm{\Sigma}_h}&=\text{Id},\quad{Q}_h|_{\bm{V}_h}=\text{Id},\label{preserv}\\
    \nabla\cdot{\pi}_h\bm{\tau}&={Q}_h\nabla\cdot\bm{\tau}\quad\text{ for }\bm{\tau}\in\bm{\Sigma},\label{compih}\\
    \|{\pi}_h\bm{\tau}\|&\lesssim\|\bm{\tau}\|\quad\text{ for }\bm{\tau}\in[L^2(\Omega)]^d,\label{bdpih}
\end{align}
\end{subequations}
where $\text{Id}$ is the identity mapping.
Then with the help of ${\pi}_h$, $Q_h$, we obtain the following lemma for estimating the decomposition component living in $\bm{Z}_h^\perp.$ \begin{lemma}\label{disPoin}
It holds that
\begin{equation*}
\|\bm{\tau}_h\|_\CA\lesssim\|\nabla\cdot\bm{\tau}_h\|_{-1},\quad\bm{\tau}_h\in\bm{Z}_h^\perp.  
\end{equation*}
\end{lemma}
\begin{proof}
Let $\bm{\phi}\in[H_0^1(\Omega)]^d$ be the weak solution to $$\Delta\bm{\phi}=\nabla\cdot\bm{\tau}_h.$$
Then $\bm{\tau}:=\nabla\bm{\phi}$ satisfies $\nabla\cdot\bm{\tau}=\nabla\cdot\bm{\tau}_h.$ Using it and \eqref{compih}, \eqref{preserv}, we obtain
\begin{align*}
&\nabla\cdot(\bm{\tau}_h-\pi_h\bm{\tau})=\nabla\cdot\bm{\tau}_h-{Q}_h\nabla\cdot\bm{\tau}\\
&\quad=\nabla\cdot\bm{\tau}_h-{Q}_h\nabla\cdot\bm{\tau}_h=\bm{0},
\end{align*}
i.e., $\bm{\tau}_h-\pi_h\bm{\tau}\in\bm{Z}_h$. Recall that $\|\cdot\|_{-1}$ denotes the $H^{-1}(\Omega)$-norm. The \emph{natural} $H^1$-elliptic regularity of $\Delta\bm{\phi}=\nabla\cdot\bm{\tau}_h$ yields
\begin{equation}\label{ellipticreg}
    \|\bm{\tau}\|\leq\|\bm{\phi}\|_1\lesssim\|\nabla\cdot\bm{\tau}_h\|_{-1}.
\end{equation} Hence using $\bm{\tau}_h-\pi_h\bm{\tau}\in\bm{Z}_h$, $\bm{\tau}_h\in\bm{Z}_h^\perp$, \eqref{ellipticreg}, and \eqref{bdpih}, we have
\begin{equation*}
    \begin{aligned}
    \|\bm{\tau}_h\|_\CA^2&=\Lambda(\bm{\tau}_h-\pi_h\bm{\tau},\bm{\tau}_h)+(\CA\pi_h\bm{\tau},\bm{\tau}_h)\\
    &=(\CA\pi_h\bm{\tau},\bm{\tau}_h)\leq\|\pi_h\bm{\tau}\|_\CA\|\bm{\tau}_h\|_\CA\\
    &\lesssim\|\bm{\tau}\|\|\bm{\tau}_h\|_\CA\lesssim\|\nabla\cdot\bm{\tau}_h\|_{-1}\|\bm{\tau}_h\|_\CA.
    \end{aligned}
\end{equation*}
The proof is complete.\qed\end{proof}
\begin{remark}
It follows from the discrete abstract Poincar\'e inequality (see \cite{ArnoldFalkWinther2010}, Theorem 3.6) that
$$\|\bm{\tau}_h\|_\CA\lesssim\|\nabla\cdot\bm{\tau}_h\|\text{ for }\bm{\tau}_h\in\bm{Z}_h^\perp.$$
Hence Lemma \ref{disPoin} is an improved discrete Poincar\'e inequality. The improvement is achieved by utilizing the special property of the divergence operator.
\end{remark}

We emphasize that the improved estimate of $\bm{\xi}_h$ is dependent on mesh structure, type of finite elements, and quite technical. For simplicity of presentation, first let $\bm{\Sigma}_h\times\bm{V}_h$ be based on the lowest order RT element \eqref{RT} with $r=0$ in $\mathbb{R}^2$ although the estimates can be generalized in several ways, see the end of this section. 
Given a scalar-valued function $v$ and a vector-valued function $\bm{v}=(v_1,v_2)^\top$, let
\begin{align*}
    \nabla^\perp v:=(-\partial_{x_2}{v},\partial_{x_1}{v})^\top,\quad \bm{v}^\perp:=(-{v_2},v_1)^\top.
\end{align*}
The row-wise rotational gradient or curl is defined as
$$\nabla^\perp\bm{v}=\left[\begin{matrix}(\nabla^\perp v_1)^\top\\(\nabla^\perp v_2)^\top\end{matrix}\right].$$
Introducing the following nodal element spaces
\begin{align*}
    {W}_h&=\{w\in {H}^{1}(\Omega): {w}|_K\in\mathcal{P}_1(K)\text{ for }K\in\Th\},\\
    \bm{W}_h&=\{\bm{w}\in [W_h]^2:~\int_{\Omega}\nabla\cdot \bm{w}^\perp dx=0\},
\end{align*}
we obtain a two-dimensional discrete sequence 
\begin{equation}\label{discomplex}
    \begin{aligned}
     \bm{W}_h\xrightarrow{\nabla^\perp}\bm{\Sigma}_{h}\xrightarrow{\nabla\cdot}\bm{V}_h\rightarrow 0.
    \end{aligned}
\end{equation}
The inclusion $\nabla^\perp\bm{W}_h\subset\bm{\Sigma}_h$follows from $\int_\Omega\nabla\cdot \bm{w}_h^\perp dx=-\int_\Omega\text{Tr}\nabla^\perp\bm{w}_hdx=0$ $\forall\bm{w}_h\in\bm{W}_h$. 

The supercloseness estimate of $\bxih$ does not hold on general unstructured grids. For elliptic equations discretized by RT elements, the author Y.~Li et al.~in \cite{Li2018SINUM,BankLi2019} derived several supercloseness estimates on the vector unknown in mixed methods  under certain mildly structured grids. Readers are referred to  \cite{Duran1990,EwingLiuWang1999,Brandts1994,Brandts2000} for similar results on Poisson's equation under rectangular, $h^2$-uniform quadrilateral, and uniform triangular grids. To avoid lengthy descriptions of different mesh structures, we focus on the following piecewise uniform grids. 
\begin{figure}[tbhp]
\centering
\includegraphics[width=12cm,height=5.0cm]{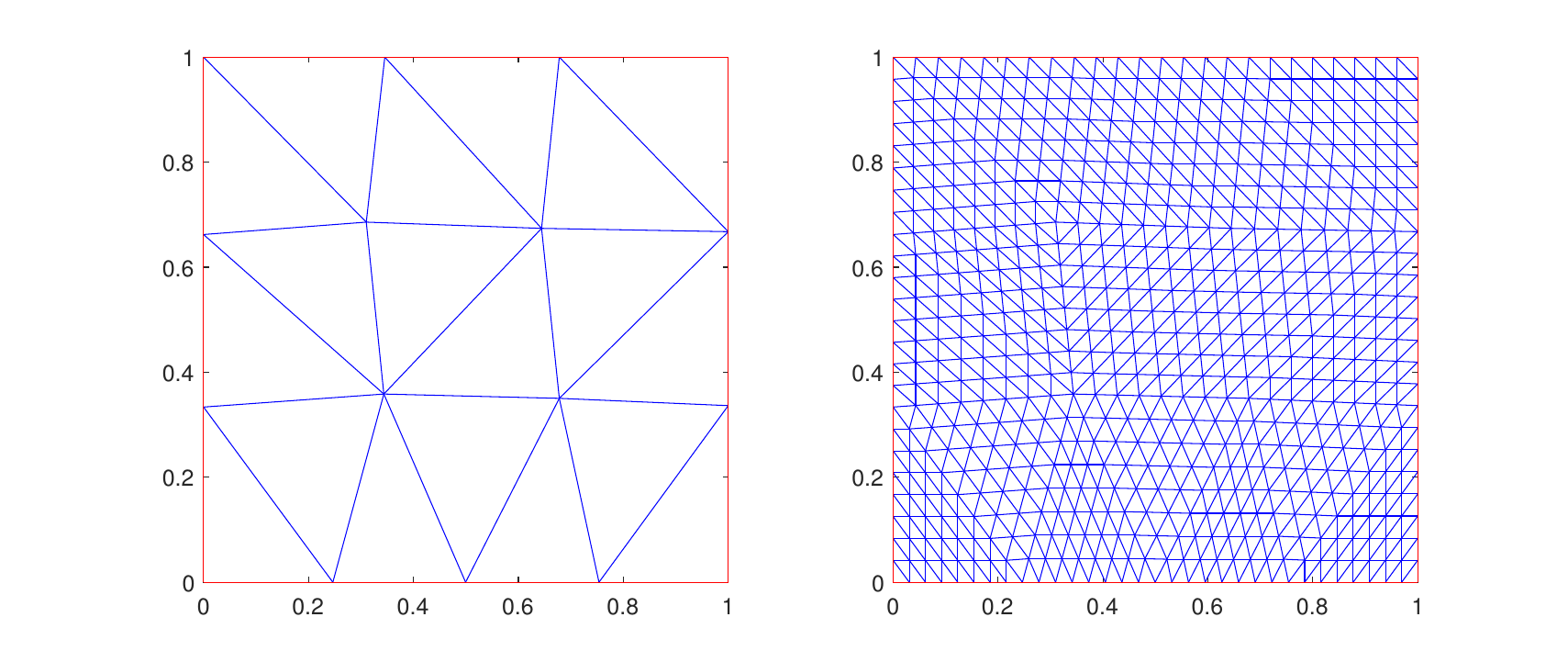}
\caption{Piecewise uniform grid on a square}
\label{pwuniform}
\end{figure}
\begin{definition}
We say $\Th$ is a uniform grid provided each pair of adjacent triangles (two triangles sharing a common edge) form an exact parallelogram. Let $\{\Omega_i\}_{i=1}^N$ be a fixed polygonal partition of  $\Omega.$ We say $\Th$ is a piecewise uniform grid provided $\Th$ is aligned with  $\{\Omega_i\}_{i=1}^N$ and  $\mathcal{T}_{h}|_{\Omega_i}$ is a uniform grid for each $1\leq i\leq N$.
\end{definition}
For instance, let $\Omega=[0,1]^2$ be the unit square that is split into the 19 triangles $\{\Omega_i\}_{i=1}^{19}$ given in Figure \ref{pwuniform} (left). After 3 consecutive uniform quad-refinement, we obtain the triangulation in Figure \ref{pwuniform} (right), which is a piecewise uniform grid (w.r.t.~$\{\Omega_i\}_{i=1}^{19}$).
The piecewise uniform mesh structure is only used in the next technical lemma.
\begin{lemma}\label{varerror}
Let $\bm{D}\in\mathbb{R}^{2\times2}$ be a constant matrix. Let $\widetilde{\Pi}_h$ be the canonical interpolation onto  $\widetilde{\bm{\Sigma}}_h$, the lowest order RT element space. For a piecewise uniform $\Th$ and $\bm{v}\in W^2_\infty(\Omega)$, $w_h\in W_h,$ it holds that
\begin{equation*}
    (\bm{\bm{D}}(\bm{v}-\widetilde{\Pi}_h\bm{v}),\nabla^\perp w_h)\lesssim h^2|\log h|^\frac{1}{2}\|\bm{v}\|_{W^2_\infty(\Omega)}\|\nabla^\perp w_h\|.
\end{equation*}
\end{lemma}
In \cite{Li2018SINUM}, one of the author proved Lemma \ref{varerror} on mildly structured grids when $\bm{D}=\bm{I}$. For the general anisotropic $\bm{D}$, we could not find a complete proof in literature. In the appendix, we give a detailed proof of Lemma \ref{varerror} using the technique in Theorem 3.2 of \cite{Brandts1994} and Lemma 3.4 of \cite{HuMaMa2021}. Now we present a supercloseness estimate for $\bm{\xi}_h$, the second main result, in the next theorem.
\begin{theorem}\label{supersigma}
Let $\Omega$ be simply-connected and $\Th$ be a piecewise uniform grid.
For \eqref{mixed} using the lowest order RT element and sufficiently small $h$, it holds that
$$\|\bxih\|\lesssim h^2|\log h|^\frac{1}{2}(\|\bm{\sigma}\|_{W^2_\infty(\Omega)}+\|\bu\|_2).$$
\end{theorem}
\begin{proof}
When $\Omega$ is simply connected,
\eqref{discomplex} is an exact sequence, i.e., $\nabla^\perp\bm{W}_h=\bm{Z}_h.$ It then follows from the exactness and the discrete Helmholtz decomposition \eqref{decomp} that 
\begin{equation}
    \begin{aligned}
    \bm{\xi}_h=\nabla^\perp\bm{\phi}_h\oplus_\Lambda\bm{\eta}_h,
    \end{aligned}
\end{equation}
where $\bm{\phi}_h\in\bm{W}_h,$  $\bm{\eta}_h\in\bm{Z}_h^\perp.$ Using Lemma \ref{disPoin} and $\nabla\cdot\nabla^\perp\bm{\phi}_h=\bm{0}$, we obtain
\begin{equation}\label{Zperppart}
    \|\bm{\eta}_h\|_\CA\lesssim\|\nabla\cdot\bm{\eta}_h\|_{-1}=\|\nabla\cdot\bm{\xi}_h\|_{-1}.
\end{equation}
It remains to estimate $\|\nabla^\perp\bm{\phi}_h\|_\CA$. Note that $$(\CA\nabla^\perp\bm{\phi}_h,\bm{\eta}_h)=\Lambda(\nabla^\perp\bm{\phi}_h,\bm{\eta}_h)=0.$$ Then using the orthogonality, \eqref{error1}, and  \eqref{adjA}, we have
\begin{equation*}
    \begin{aligned}
    &\|\nabla^\perp\pih\|_\CA^2=(\ma(\Pi_h\bm{\sigma}-\bm{\sigma}_h),\nabla^\perp\pih)\\
    &=(\ps-\s,\CA\nabla^\perp\pih)=(\tps-\s,\CA\nabla^\perp\pih)\\
    &=-(\s-\tps,\nabla^\perp\pih)-\frac{1}{2}(\s-\tps,(\nabla\cdot\bm{\phi}^\perp_h)\bm{I}).
    \end{aligned}
\end{equation*}
Recall that $\bm{\sigma}_i$ denotes the $i$-th \emph{row} of $\bm{\sigma}.$ Let $\bm{e}_1=(1,0)$, $\bm{e}_2=(0,1)$, and $\bm{\phi}_h=(\phi_{h,1},\phi_{h,2})^\top$. It follows from the previous equation and direct calculation that
\begin{equation}\label{Zpartinter}
    \begin{aligned}
    &\|\nabla^\perp\pih\|_\CA^2=\sum_{i=1}^2\left\{-(\s_i-\widetilde{\Pi}_h\bm{\sigma}_i,\nabla^\perp{\phi}_{h,i})\right.\\
    &\quad\left.+\frac{1}{2}(\s_i-\tps_i,(\partial_{x_1}{\phi}_{h,2})\bm{e}_i)-\frac{1}{2}(\s_i-\tps_i,(\partial_{x_2}{\phi}_{h,1})\bm{e}_i)\right\}\\
    &\quad=\sum_{i,j=1}^2(\bm{D}_{ij}(\s_i-\widetilde{\Pi}_h\bm{\sigma}_i)^\top,\nabla^\perp\phi_{h,j}),
    \end{aligned}
\end{equation}
where 
$$\bm{D}_{11}=\begin{pmatrix}-\frac{1}{2}&0\\0&-1\end{pmatrix},~\bm{D}_{12}=\begin{pmatrix}0&0\\ \frac{1}{2}&0\end{pmatrix},~\bm{D}_{21}=\begin{pmatrix}0&\frac{1}{2}\\0&0\end{pmatrix},~\bm{D}_{22}=\begin{pmatrix}-1&0\\0&-\frac{1}{2}\end{pmatrix}.$$ 
Then using \eqref{Zpartinter}, Lemma \ref{varerror} and \eqref{rb} with $\bm{\tau}=\nabla^\perp\bm{\phi}_h$, we have 
\begin{equation}\label{Zpart}
\begin{aligned}
\|\nabla^\perp\pih\|_\CA^2&\lesssim h^2|\log h|^\frac{1}{2}\|\bm{\sigma}\|_{W^2_\infty(\Omega)}\|\nabla^\perp\bm{\phi}_{h}\|\\
&\lesssim h^2|\log h|^\frac{1}{2}\|\bm{\sigma}\|_{W^2_\infty(\Omega)}\|\nabla^\perp\bm{\phi}_{h}\|_\CA.
\end{aligned}
\end{equation}
Combining \eqref{rb} with \eqref{Zperppart}, \eqref{Zpart}, \eqref{divxin1}, \eqref{ratesuperu}, we obtain
\begin{equation*}
\begin{aligned}
\|\bm{\xi}_h\|&\lesssim\|\bm{\xi}_h\|_\CA+\|\nabla\cdot\bm{\xi}_h\|_{-1}\\
&\lesssim\|\nabla^\perp\bm{\phi}_h\|_\CA+\|\bm{\eta}_h\|_\CA+\|\nabla\cdot\bm{\xi}_h\|_{-1}\\
&\lesssim h^2|\log h|^\frac{1}{2}\big(\|\bm{\sigma}\|_{W^2_\infty(\Omega)}+\|\bm{u}\|_2\big).
\end{aligned}
\end{equation*}
The proof is complete.
\qed\end{proof}

The supercloseness estimate on $\|\bm{\xi}_h\|$ can be generalized on rectangular meshes. Throughout the rest of this section, let $\Th$ be a rectangular mesh, 
\begin{align*}
    {W}_h&=\left\{w_h\in {H}^{1}(\Omega): {w}_h|_K\in\mathcal{Q}_{r+1}(K)\text{ for }K\in\Th\right\},\\
    \bm{W}_h&=\left\{\bm{w}_h\in [W_h]^2:~\int_{\Omega}\nabla^\perp \bm{w}_hdx=0\right\},
\end{align*}
and $\bm{\Sigma}_h\times\bm{V}_h$ be based on the rectangular RT element \eqref{rectRT} with $d=2$. When $\Omega$ is simply-connected, we still have the discrete exact sequence \eqref{discomplex}. Similarly to the case of triangular grids, we need the uniform mesh structure. \begin{definition}
We say $\Th$ is a uniformly rectangular mesh if all the rectangles of $\Th$ are of the same shape and size. 
\end{definition}
For $w_h\in{W}_h$ with $(\nabla^\perp w_h)\cdot\bm{n}|_{\partial\Omega}=0$, Theorem 5.1 of \cite{EwingLiuWang1999} implies
\begin{equation*}
    (\bm{\bm{D}}(\bm{v}-\widetilde{\Pi}_h\bm{v}),\nabla^\perp w_h)\lesssim h^{r+2}\|\bm{v}\|_{r+2}\|\nabla^\perp w_h\|,
\end{equation*}
where $\bm{v}\in[H^{r+2}(\Omega)]^2$. The previous estimate is not true when $(\nabla^\perp w_h)\cdot\bm{n}\neq 0$ on $\partial\Omega$. As claimed in the remark in Example 6.2 of \cite{EwingLiuWang1999},  it holds that
\begin{equation}\label{varrect}
    (\bm{\bm{D}}(\bm{v}-\widetilde{\Pi}_h\bm{v}),\nabla^\perp w_h)\leq  Ch^{r+1.5}\|\nabla^\perp w_h\|,\quad w_h\in W_h
\end{equation}
for some absolute constant $C$ independent of $h$.  
Following the same argument in the proof of Theorem \ref{supersigma} and replacing Lemma \ref{varerror} by \eqref{varrect}, we actually obtain the supercloseness estimate
$$\|\Pi_h\bm{\sigma}-\bm{\sigma}_h\|=\mathcal{O}(h^{r+1.5}).$$
Similar estimates may hold on $h^2$-uniform quadrilateral grids described in \cite{EwingLiuWang1999}.
\section{Postprocessing}\label{secpost} We have derived supercloseness estimates in Lemma \ref{superu} and Theorem \ref{supersigma}. However, those results can not be directly used because $P_h\bm{u}$ and $\Pi_h\bm{\sigma}$ are not known in practice. To extract superconvergence information from the smallness of $\bm{e}_h$ and $\bm{\xi}_h$, one may design easy-to-compute postprocessed solutions $\bm{u}_h^*$ and $\bm{\sigma}_h^*$. For instance, 
following the idea of \cite{Stenberg1991}, a element-by-element postprocessing procedure for $\bm{u}_h$ in the pseudostress-velocity formulation of the Stokes equation is proposed in \cite{CKP2011}. In particular, let
$$\mathcal{P}_{r+1}(K)/\mathcal{P}_{r}(K):=\{\bm{v}_h\in\mathcal{P}_{r+1}(K): (\bm{v}_h,\bm{w}_h)_K=0\text{ for all }\bm{w}_h\in\mathcal{P}_r(K)\}.$$
where $(\cdot,\cdot)_K$ is the $L^2(K)$-inner product. Note that the pressure could be recovered as $$p_h:=-\frac{1}{d}\text{Tr}\bm{\sigma}_h\approx p.$$ 
For each $K\in\Th,$
the postprocessed solution  $\bm{u}_h^*|_K\in\mathcal{P}_{r+1}(K)$ is determined by
\begin{equation}
    \begin{aligned}
    \nu(\nabla\bm{u}_h^*,\nabla\bm{v}_h)_K&=(\bm{\sigma}_h,\nabla\bm{v}_h)_K+(p_h,\nabla\cdot\bm{v}_h)_K,~\bm{v}_h\in\mathcal{P}_{r+1}(K)/\mathcal{P}_{r}(K),\\
    (\bm{u}_h^*,\bar{\bm{v}}_h)_K&=(\bm{u}_h,\bar{\bm{v}}_h)_K,\quad\bar{\bm{v}}_h\in\mathcal{P}_{r}(K).
    \end{aligned}
\end{equation}
Since the analysis of the mapping $\bm{u}_h\mapsto\bm{u}_h^*$ is independent of equations, we can combine the analysis in Theorem 4.1 of \cite{CKP2011} for the Stokes equation with the supercloseness on $\|P_h\bm{u}-\bm{u}_h\|$ in Theorem \ref{superu} to obtain the following postprocessing superconvergence estimate. 
\begin{theorem}\label{recoveryu}
For sufficiently small $h$, it holds that
\begin{equation*}
    \|\bm{u}-\bm{u}^*_h\|\lesssim h^{r+2}\big(\|\bm{u}\|_{r+2}+|\bm{\sigma}|_{r+1}+|\nabla\cdot\bm{\sigma}|_{r+1}\big).
\end{equation*}
\end{theorem}
\begin{proof}
Let $\bar{\bm{u}}_h$ be the $L^2$-projection of $\bm{u}$ onto the space 
\begin{equation*}
    \left\{\bm{v}_h: \bm{v}_h|_K\in[\mathcal{P}_{r+1}(K)]^d\text{ for all }K\in\mathcal{T}_h\right\}.
\end{equation*}
It has been shown in the proof of Theorem 4.1 in \cite{CKP2011} that 
\begin{equation}\label{ipterm}
\begin{aligned}
    &h^{-1}\|(\bm{I}-P_h)(\bar{\bm{u}}_h-\bm{u}_h^*)\|\lesssim h^{-1}\|P_h(\bar{\bm{u}}_h-\bm{u}_h^*)\|\\
    &\quad+\|\bm{\sigma}-\bm{\sigma}_h\|+\left(\sum_{K\in\Th}\|\nabla(\bar{\bm{u}}_h-\bm{u})\|^2_{L^2(K)}\right)^\frac{1}{2}.
    \end{aligned}
\end{equation}
Using the triangle inequality, $P_h\bar{\bm{u}}_h=P_h\bm{u},$  $P_h{\bm{u}}^*_h=\bm{u}_h$, and \eqref{ipterm},  we obtain
\begin{equation*}
    \begin{aligned}
    \|\bm{u}-\bm{u}^*_h\|&\leq\|\bm{u}-\bar{\bm{u}}_h\|+\|(\bm{I}-P_h)(\bar{\bm{u}}_h-\bm{u}^*_h)\|+\|P_h(\bar{\bm{u}}_h-\bm{u}^*_h)\|\\
    &\lesssim h^{r+2}|\bm{u}|_{r+2}+\|P_h\bm{u}-\bm{u}_h\|+h\|\bm{\sigma}-\bm{\sigma}_h\|.
    \end{aligned}
\end{equation*}
We finally conclude the proof from the previous estimate with \eqref{ratesuperu} and \eqref{convrateb}.
\qed\end{proof}

Postprocessing technique on the scalar variable in mixed methods for Poisson's equation can be found in e.g.,  \cite{BrezziDouglasMarini1985,BrambleXu1989,Stenberg1991,LovadinaStenberg2006}.

The postprocessing procedure $\bm{\sigma}_h\mapsto\bm{\sigma}_h^*$ can be derived from existing postprocessing operator $\widetilde{R}_h$. In particular,  $\widetilde{R}_h: \widetilde{\bm{\Sigma}}_h\rightarrow Y\subset[L^2(\Omega)]^2$ is a linear mapping onto the space $Y$ of suitable piecewise polynomials and satisfies 
\begin{subequations}\label{Rtilde}
\begin{align}
\widetilde{R}_h\bm{c}&=\bm{c},\quad\bm{c}\in\mathbb{R}^2,\label{preI}\\
\|\widetilde{R}_h\bm{v}\|&\lesssim \|\bm{v}\|,\quad\bm{v}\in\widetilde{\bm{\Sigma}}_h,\label{Rtildebd}\\
\|\bm{\tau}-\widetilde{R}_h\widetilde{\Pi}_h\bm{\tau}\|&\lesssim h^{1+\alpha}\|\bm{\tau}\|_{W^2_\infty(\Omega)},\label{Rtildeapprox}
\end{align}
\end{subequations}
for some positive constant $\alpha$ and sufficiently smooth $\bm{\tau}.$ When $\widetilde{\bm{\Sigma}}_h$ is the lowest order RT element space, $\widetilde{R}_h$ that satisfies \eqref{Rtilde} is given in e.g., \cite{Brandts1994,BankLi2019}. The simple nodal or edge averaging \cite{ZZ1987,Brandts1994} and superconvergent patch recovery \cite{ZZ1992a,XuZhang2004,BankLi2019} are also possible choices.  

For $\bm{\tau}\in\bm{\Sigma}_h$, let 
$$R_h\bm{\tau}:=\widetilde{R}_h\bm{\tau}-\frac{1}{2|\Omega|}\left(\int_\Omega\text{Tr}\widetilde{R}_h\bm{\tau}dx\right)\bm{I},$$
where $\widetilde{R}_h$ is applied to each row of $\bm{\tau}$. We have the following super-approximation result of ${R}_h.$
\begin{lemma}\label{superapprox}
Assume \eqref{Rtilde} holds. For $\bm{\tau}\in\bm{\Sigma}\cap[W^2_\infty(\Omega)]^{2\times2}$, we have
$$\|\bm{\tau}-{R}_h{\Pi}_h\bm{\tau}\|\lesssim h^{1+\alpha}\|\bm{\tau}\|_{W^2_\infty(\Omega)}.$$
\end{lemma}
\begin{proof}
\eqref{preI} implies $R_h\bm{I}=\bm{O}$  and thus ${R}_h{\Pi}_h\bm{\tau}={R}_h\widetilde{\Pi}_h\bm{\tau}.$ Then 
\begin{equation}\label{superapproxinter}
    \begin{aligned}
    &\bm{\tau}-{R}_h{\Pi}_h\bm{\tau}=\bm{\tau}-{R}_h\widetilde{\Pi}_h\bm{\tau}\\
    &\quad=\bm{\tau}-\widetilde{R}_h\widetilde{\Pi}_h\bm{\tau}-\frac{1}{2|\Omega|}\int_\Omega\text{Tr}(\bm{\tau}-\widetilde{R}_h\widetilde{\Pi}_h\bm{\tau})dx,
    \end{aligned}
\end{equation}
where we used $\int_\Omega\text{Tr}\bm{\tau}dx=0$. Lemma \ref{superapprox} then follows from \eqref{superapproxinter}, the Cauchy--Schwarz inequality, and \eqref{Rtildeapprox}.
\qed\end{proof}

Combining Lemma \ref{superapprox} and Theorem \ref{supersigma}, we obtain the following  superconvergence estimate by postprocessing.
\begin{theorem}\label{recoverysigma}
Let $\bm{\sigma}_h$ be the solution to \eqref{mixed} based on the lowest order RT element and $\bm{\sigma}^*_h:=R_h\bm{\sigma}_h$. Let the assumptions in Theorem \ref{supersigma} and Lemma \ref{superapprox} hold. We have 
\begin{equation*}
    \|\bm{\sigma}-\bm{\sigma}^*_h\|\lesssim\max(h^2|\log h|^\frac{1}{2},h^{1+\alpha})\big(\|\bm{\sigma}\|_{W^2_\infty(\Omega)}+\|\bm{u}\|_2\big).
\end{equation*}
\end{theorem}
\begin{proof}
Using the triangle inequality, Theorem \ref{supersigma}, and Lemma \ref{superapprox}, we have
\begin{equation*}
    \begin{aligned}
    \|\bm{\sigma}-\bm{\sigma}^*_h\|&\leq\|\bm{\sigma}-R_h\Pi_h\bm{\sigma}\|+\|R_h(\Pi_h\bm{\sigma}-\bm{\sigma}_h)\|\\
    &\lesssim\max(h^2|\log h|^\frac{1}{2},h^{1+\alpha})\big(\|\bm{\sigma}\|_{W^2_\infty(\Omega)}+\|\bm{u}\|_2\big),
    \end{aligned}
\end{equation*}
which completes the proof.
\qed\end{proof}
\begin{remark}
One could reconstruct accurate numerical symmetric stress and pressure from the superconvergent pseudostress $\bm{\sigma}_h^*$. In fact, let 
\begin{align*}
    \bm{\sigma}^{S*}_h:=\frac{1}{2}(\bm{\sigma}_h^*+\bm{\sigma}_h^{*\top}),\quad p_h^*:=-\frac{1}{d}{\rm Tr}\bm{\sigma}_h^*.
\end{align*}
We conclude that $\bm{\sigma}^{*S}_h$ superconverges to the symmetric stress $\bm{\sigma}^{S}$ and $p_h^*$ superconverges
to $p$ the pressure from Theorem \ref{recoverysigma} and the elementary inequalities
\begin{align*}
    \|\bm{\sigma}^S-\bm{\sigma}^{*S}_h\|&\leq\frac{1}{2}\|\bm{\sigma}-\bm{\sigma}^*_h\|+\frac{1}{2}\|\bm{\sigma}^\top-\bm{\sigma}^{*\top}_h\|=\|\bm{\sigma}-\bm{\sigma}^*_h\|,\\
    \|p-p^{*}_h\|&\leq\frac{1}{d}\|{\rm Tr}(\bm{\sigma}-\bm{\sigma}^*_h)\|\leq\frac{1}{\sqrt{d}}\|\bm{\sigma}-\bm{\sigma}^*_h\|.
\end{align*}
\end{remark}
\section{Experiments}\label{secexp}
In this section, the postprocessing procedure $\bm{\sigma}_h\mapsto\bm{\sigma}_h^*$ in Theorem \ref{recoverysigma} is based on the polynomial preserving recovery $\widetilde{R}_h$ for the lowest order RT element described in \cite{BankLi2019}. Roughly speaking, each row of $\bm{\sigma}_h^*$ is constructed as a continuous piecewise linear vector-valued polynomial with nodal values determined by least-squares fitted linear local polynomial on vertex patches. For linear Oseen equations, we set the viscosity $\nu$ to be $1$ and verify the a priori and superconvergence error estimates. The adaptivity performance of the error estimator base on $\bm{\sigma}_h^*$, $\bm{u}_h^*$ is also under investigation. In the end, the proposed postprocessing is tested in the Kovasznay flow from the Navier--Stokes equation.

\begin{table}[tbhp]
\caption{Convergence history of the lowest order $RT$ element}
\centering
\begin{tabular}{|c|c|c|c|c|c|c|}
\hline
nt & $\|\bm{u}-\bm{u}_{h}\|$&$\|\eh\|$
 &$\|\bm{u}-\bm{u}^*_{h}\|$
&  $\|\bm{\sigma}-\bm{\sigma}_{h}\|$ &  $\|\bm{\xi}_h\|$ &  $\|\bm{\sigma}-\bm{\sigma}_{h}^*\|$ \\
\hline
             19   &3.210e-1	&3.851e-2	&9.334e-2&1.372&2.505e-1&2.113\\
             76     &1.620e-1	&1.027e-2	&2.438e-2&6.849e-1&7.172e-2&6.108e-1\\
             304    &8.121e-2	&2.621e-3	&6.176e-3&3.418e-1&1.957e-2&1.724e-1\\
             1216      &4.063e-2	&6.598e-4	&1.550e-3&1.707e-1&5.242e-3&4.353e-2\\
             4864       &2.032e-2	&1.653e-4	&3.880e-4&8.530e-2&1.390e-3&1.088e-2\\
             19456       &1.016e-2	&4.136e-5	&9.703e-5&4.270e-2&3.659e-4&2.719e-3\\
\hline
order &9.990e-1&1.990&1.994&1.001&1.904&1.961\\
\hline
\end{tabular}
\label{RTtab}
\end{table}

\begin{table}[tbhp]
\caption{Convergence history of the lowest order $BDM$ element}
\centering
\begin{tabular}{|c|c|c|c|c|c|}
\hline
nt & $\|\bm{u}-\bm{u}_{h}\|$&$\|\eh\|$
 &$\|\bm{u}-\bm{u}^*_{h}\|$
&  $\|\bm{\sigma}-\bm{\sigma}_{h}\|$ &  $\|\bm{\xi}_h\|$\\
\hline
             19   &3.192e-1	&1.779e-2	&6.316e-02&4.480e-1&4.755e-1\\
             76     &1.618e-1	&6.026e-3	&1.645e-02 &1.249e-1&1.180e-1\\
             304    &8.118e-2	&1.636e-3	&4.155e-03&3.216e-2&2.928e-2\\
             1216      &4.062e-2	&4.181e-4	&1.042e-03&8.104e-3&7.288e-3\\
             4864       &2.032e-2	&1.051e-4	&2.606e-04 &2.031e-3&1.818e-3\\
             19456       &1.016e-2	&2.631e-5	&6.515e-05
  &5.081e-4&4.541e-4\\
\hline
order &9.986e-1&1.964&1.996&1.987&2.005\\
\hline
\end{tabular}
\label{BDMtab}
\end{table}
\subsection{A priori convergence} Consider the Oseen equation \eqref{Oseen1} on the unit square  $\Omega=[0,1]^2$  with the smooth solutions \begin{align*}
    \bm{u}&=\begin{pmatrix}\sin(\pi(x_1+x_2))\\-\sin(\pi(x_1+x_2))\end{pmatrix},\\
    p&=x_1+x_2-1.
\end{align*} 
We set $c=0$,  $\bm{b}=(\cos(x_2),\sin(x_1))^\top$,  $\bm{g}=\bm{u}|_{\partial\Omega}=\bm{0}$. $\bm{f}$ is computed from $\bm{u}$ and $\bm{b}$. We start with the initial partition in Figure \ref{pwuniform} (left). A sequence of piecewise uniform meshes is obtained by uniform quad-refinement, i.e., dividing each triangle into four similar subtriangles by connecting the midpoints of each edge. Numerical results are presented in Tables \ref{RTtab} and \ref{BDMtab}, where nt is short for ``number of triangles''. The order of convergence is computed from the error quantities in those tables by least squares without using the data in the first rows.

The numerical rates of convergence coincide with a priori error estimates \eqref{convrate}, \eqref{convrateBDM} and the  supercloseness estimates in \eqref{ratesuperu} and Theorem \ref{supersigma}. It is noted that for the lowest order BDM element, $\|\bm{\xi}_h\|\approx\mathcal{O}(h^2)$, which is predicted by a priori error estimates and thus not supersmall.
Since the recovery procedure in \cite{BankLi2019} provides the super-approximation rate $\alpha=1$ in \eqref{Rtildeapprox} and Lemma \ref{superapprox}, the recovery superconvergence estimate for the lowest order RT element predicted by Theorem \ref{recoverysigma} is 
$\|\bm{\sigma}-\bm{\sigma}_h^*\|=\mathcal{O}(h^2|\log h|^\frac{1}{2})$, numerically confirmed by the last column in Table \ref{RTtab}.

\begin{figure}[tbhp]
\centering
\includegraphics[width=12cm,height=5.0cm]{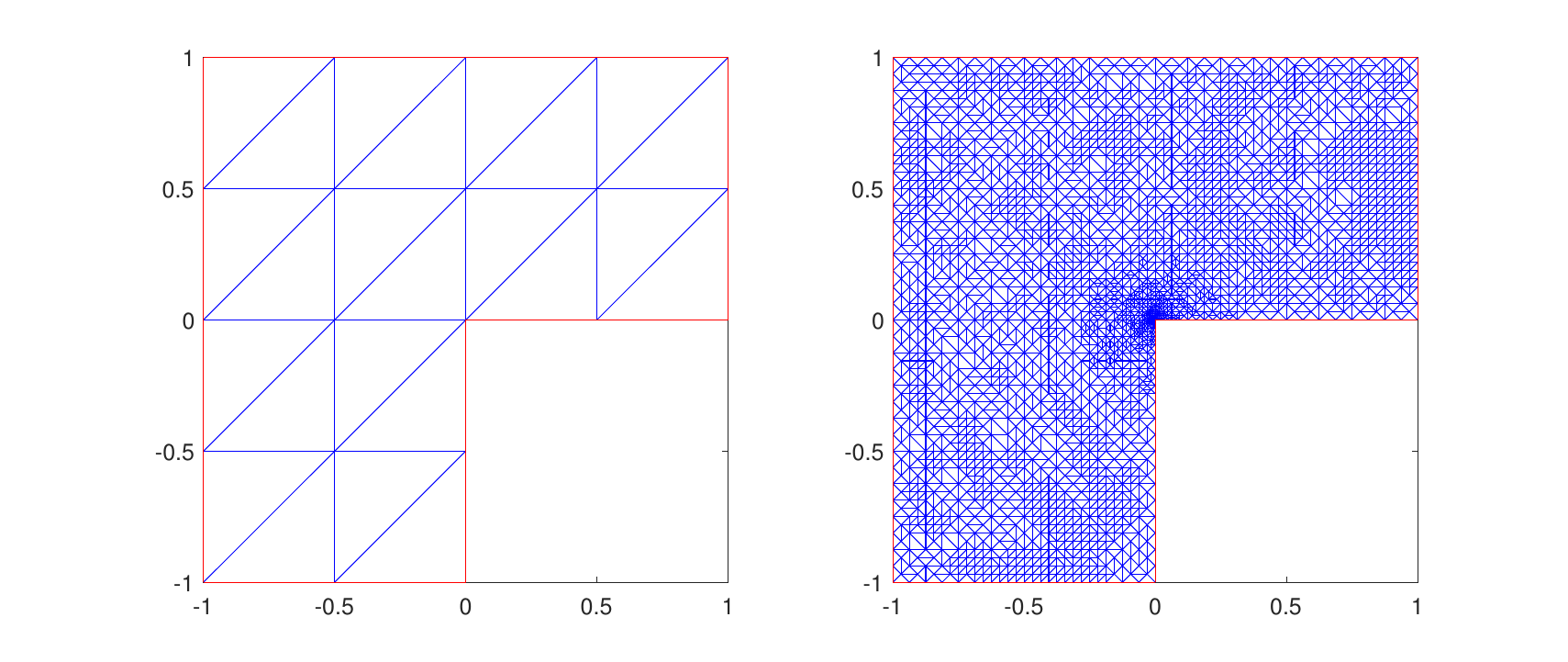}
\caption{Adaptive  mesh for the singular problem}
\label{adapt}
\end{figure}
\begin{figure}[tbhp]
\centering
\includegraphics[width=12cm,height=5.0cm]{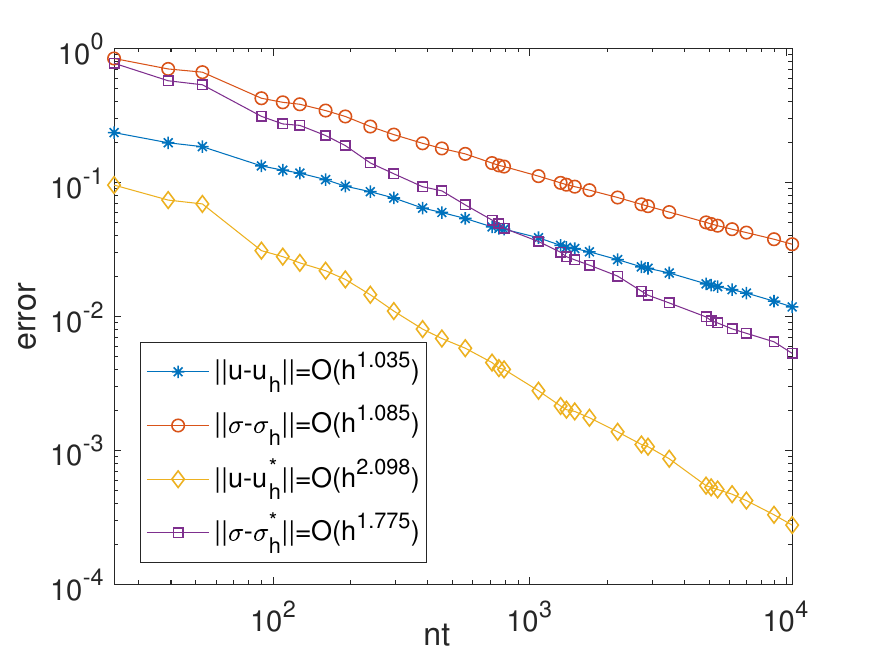}
\caption{Convergence history of the adaptive method}
\label{curve}
\end{figure}
\subsection{Adaptive mesh refinement for a non-smooth problem}\label{subsectionAMR} Consider  \eqref{Oseen1} on the L-shaped domain  $\Omega=[-1,1]^2\backslash([0,1]\times[-1,0])$ with the smooth pressure $p=x_1+x_2$ and singular velocity \begin{align*}
    \bm{u}&=\begin{pmatrix}r^\alpha\sin(\alpha\theta)\\r^\alpha\cos(\alpha\theta)\end{pmatrix},
\end{align*} 
where $\alpha=\frac{2}{3}$, $(r,\theta)$ is the polar coordinate w.r.t.~$(0,0)$.
Set $c=0$,  $\bm{b}=(1,2)^\top$ and $\bm{g}=\bm{u}|_{\partial\Omega}$. Direct calculation shows that $\nabla\cdot\bm{u}=0$ and $\bm{f}=(1,1)^\top+\bm{b}\cdot\nabla\bm{u}.$ In this experiment, we use the classical adaptive feedback loop (cf.~\cite{BabuskaRheinboldt1978,Dorfler1996,MNS2000})
$$\texttt{SOLVE}\rightarrow\texttt{ESTIMATE}\rightarrow\texttt{MARK}
\rightarrow\texttt{REFINE}$$ 
to obtain a sequence of adaptively refined grids $\{\mathcal{T}_{h_\ell}\}_{\ell\geq0}$ and numerical solutions $(\{\bm{\sigma}_{h_\ell},\bm{u}_{h_\ell})\}_{\ell\geq0}$. In particular, the algorithm starts from the initial grid $\mathcal{T}_{h_0}$ presented in Figure~\ref{adapt}(left). The module \texttt{ESTIMATE} computes the superconvergent recovery-based error indicator $$\mathcal{E}_\ell(K):=\big(\|\bm{\sigma}_{h_\ell}^*-\bm{\sigma}_{h_\ell}\|^2_{L^2(K)}+\|\bm{u}_{h_\ell}^*-\bm{u}_{h_\ell}\|^2_{L^2(K)}\big)^\frac{1}{2}$$
for each triangle $K\in\mathcal{T}_{h_\ell}$.
The module \texttt{MARK} then selects a collection of triangles 
\begin{equation}\label{mark}
\mathcal{M}_\ell:=\{K\in\mathcal{T}_\ell: \mathcal{E}_\ell(K)\geq0.7\max_{K^\prime\in\mathcal{T}_\ell}\mathcal{E}_\ell(K^\prime)\}
\end{equation}
to be refined by local quad-refinement. To remove the newly created hanging nodes, minimal number of neighboring elements of $\mathcal{M}_\ell$ are bisected and the next level triangulation $\mathcal{T}_{h_{\ell+1}} $ is generated. See Figure~\ref{adapt}(right) for an adaptively  refined triangulation.

It can be observed from Figure \ref{curve} that $(\bm{\sigma}_{h_\ell},\bm{u}_{h_\ell})$ optimally converges to $(\bm{\sigma},\bm{u})$. In addtion, the errors $\|\bm{\sigma}-\bm{\sigma}^*_{h_\ell}\|$ and $\|\bm{u}-\bm{u}^*_{h_\ell}\|$ are apparently superconvergent to 0. Let
\begin{align*}
\mathcal{E}_{\ell}&:=\big(\sum_{K\in\mathcal{T}_{h_\ell}}\mathcal{E}_{\ell}(K)^2\big)^\frac{1}{2}=\big(\|\bm{\sigma}^*_{h_\ell}-\bm{\sigma}_{h_\ell}\|^2+\|\bm{u}^*_{h_\ell}-\bm{u}_{h_\ell}\|^2\big)^\frac{1}{2},\\
E_\ell&:=\big(\|\bm{\bm{\sigma}}-\bm{\sigma}_{h_\ell}\|^2+\|\bm{u}-\bm{u}_{h_\ell}\|^2\big)^\frac{1}{2}.
\end{align*}
Due to the observed superconvergence phenomena and the triangle inequality
\begin{align*}
|\mathcal{E}_\ell-E_\ell|\leq\big(\|\bm{\bm{\sigma}}-\bm{\sigma}^*_{h_\ell}\|^2+\|\bm{u}-\bm{u}^*_{h_\ell}\|^2\big)^\frac{1}{2},
\end{align*} 
the recovery-based error estimator $\mathcal{E}_\ell$ is asymptotically exact, i.e.,
\begin{align*}
    \lim_{\ell\rightarrow\infty}\frac{\mathcal{E}_\ell}{E_\ell}=1.
\end{align*}

\begin{figure}[tbhp]
\centering
\includegraphics[width=12cm,height=5.0cm]{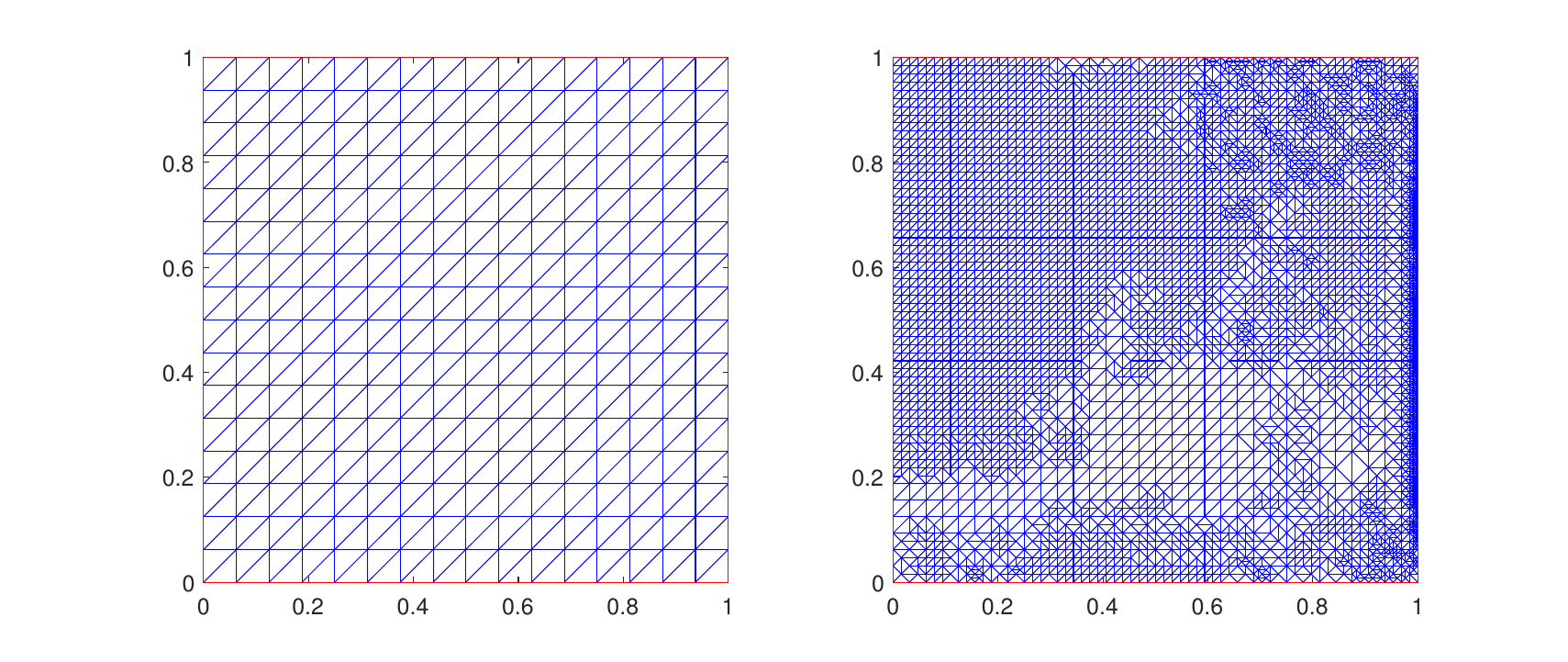}
\caption{(left) Initial grid. (right) Adaptive grid, 9683 triangles.}
\label{BoundaryLayer}
\end{figure}

\begin{figure}[tbhp]
\centering
\includegraphics[width=13cm,height=5.0cm]{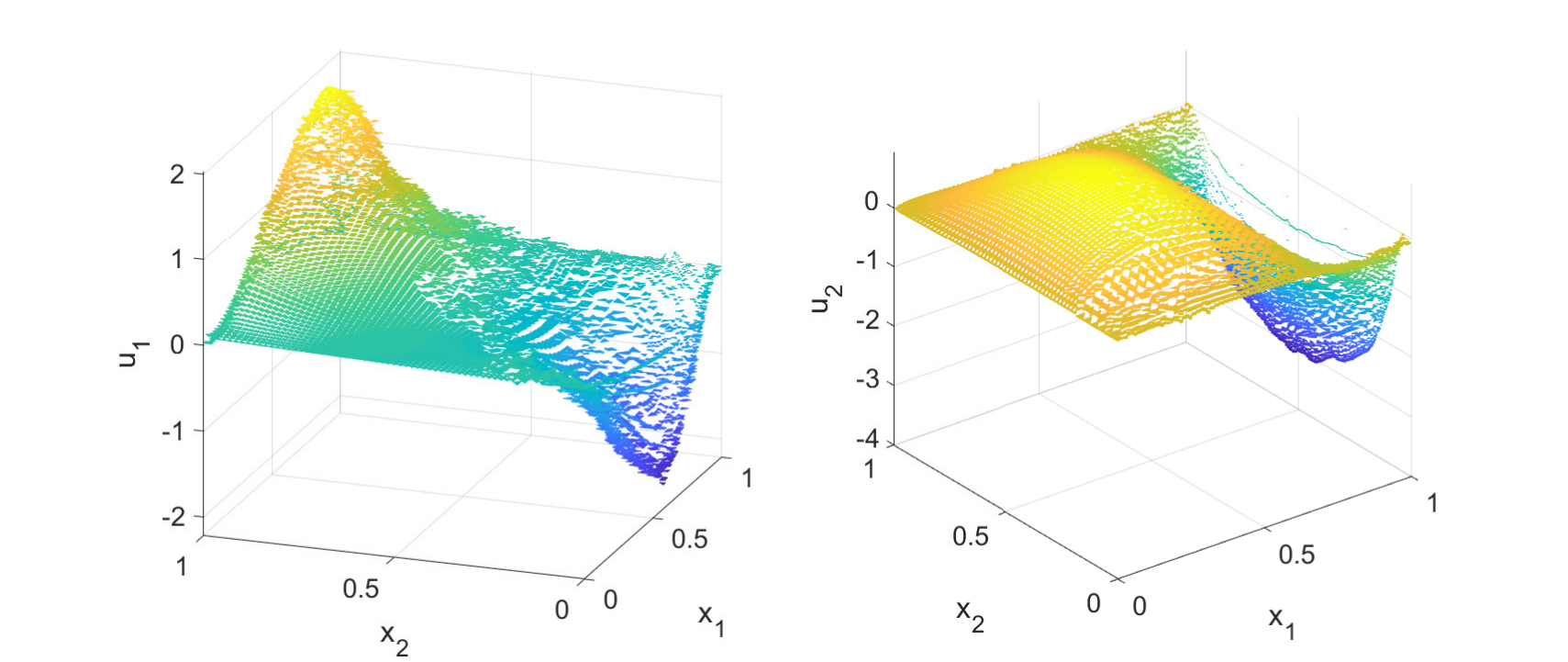}
\caption{Velocity field of the convection-dominated problem}
\label{u1u2}
\end{figure}
\subsection{Adaptive mesh refinement for dominant convection}
Consider the Oseen equation \eqref{Oseen1} on the unit square  $\Omega=[0,1]^2$ with $c=0$,  $\bm{b}=(500,1)^\top$,  $\bm{g}=\bm{0}$, $\bm{f}=5000(x_2,-x_1)^\top$.  In the end, we use the same adaptive algorithm in Subsection \ref{subsectionAMR} to solve this convection-dominated example. The initial grid is given in Figure \ref{BoundaryLayer} (left). The marked set in \eqref{mark} is replaced by
\begin{equation*}
\mathcal{M}_\ell=\{K\in\mathcal{T}_\ell: \mathcal{E}_\ell(K)\geq0.3\max_{K^\prime\in\mathcal{T}_\ell}\mathcal{E}_\ell(K^\prime)\}.
\end{equation*}
Due to the convection coefficient $\bm{b}=(500,1)^\top$, the exact solution $\bm{u}$ near the axis $x_1=1$ is rapidly changing to preserve the homogeneous Dirichlet boundary condition. It can be observed from Figures \ref{BoundaryLayer} and \ref{u1u2} that the adaptive mixed method  is able to capture the boundary layer by adaptively graded grids.

\subsection{Incompressible Navier--Stokes equation} 
Similarly to the linear Oseen equation, \eqref{NSE1} could be recast into the following  pseudostress-velocity form
\begin{equation}\label{NSE}
    \begin{aligned}
           (\mathcal{A}\bm{\sigma},\bm{\tau})+\nu(\nabla\cdot\bm{\tau},\bm{u})&=\nu\ab{\bm{g}}{\bm{\tau}\bm{n}},\quad\bm{\tau}\in\bm{\Sigma},\\ -\nu(\nabla\cdot
           \bm{\sigma},\bm{v})+(\bm{u}\cdot\mathcal{A}\bm{\sigma},\bm{v})&=\nu(\bm{f},\bm{v}),\quad\bm{v}\in\bm{V}.
    \end{aligned}
\end{equation}
Let $\Omega=[-0.5,1.5]\times[0,2]$.
The exact solution solutions are taken to be
\begin{equation}\label{exactup}
    \bm{u}=\begin{pmatrix}1-e^{\lambda x_1}\cos(2\pi x_2)\\\frac{\lambda}{2\pi}e^{\lambda x_1}\sin(2\pi x_2)\end{pmatrix},\quad p=-\frac{1}{2}e^{2\lambda x_1}+\frac{1}{8\lambda}(e^{3\lambda}-e^{-\lambda}),
\end{equation}
where $\lambda=\frac{1}{2\nu}-\sqrt{\frac{1}{4\nu^2}+4\pi^2}$, $\nu=0.025$.  In fact \eqref{exactup} is a well-known benchmark problem known as the Kovasznay flow (cf.~\cite{DiPietroErn2012,ChenLiCorinaCimbala2020}). 
In this experiment, we use the lowest order RT element $\bm{\Sigma}_h\times\bm{V}_h$ to discretize \eqref{NSE}:
\begin{equation*}
    \begin{aligned}
           (\mathcal{A}\bm{\sigma}_h,\bm{\tau})+\nu(\nabla\cdot\bm{\tau},\bm{u}_h)&=\nu\ab{\bm{g}}{\bm{\tau}\bm{n}},\quad\bm{\tau}\in\bm{\Sigma}_h,\\ -\nu(\nabla\cdot
           \bm{\sigma}_h,\bm{v})+(\bm{u}\cdot\mathcal{A}\bm{\sigma}_h,\bm{v})&=\nu(\bm{f},\bm{v}),\quad\bm{v}\in\bm{V}_h.
    \end{aligned}
\end{equation*}
The initial triangulation is a uniform partition of the square $\Omega$ with 512 right triangles. A sequence of grids is then generated by subdividing each triangles into four congruent subtriangles. 

Although our analysis is devoted to the linear Oseen equation, one could observe apparent superconvergence in both pseudostress and velocity of the Navier--Stokes equation from Table \ref{NStab}.

\begin{table}[tbhp]
\caption{Convergence history of the lowest order $RT$ element for \eqref{NSE}}
\centering
\begin{tabular}{|c|c|c|c|c|c|c|}
\hline
nt & $\|\bm{u}-\bm{u}_{h}\|$&$\|\eh\|$
 &$\|\bm{u}-\bm{u}^*_{h}\|$
&  $\|\bm{\sigma}-\bm{\sigma}_{h}\|$ &  $\|\bm{\xi}_h\|$ &  $\|\bm{\sigma}-\bm{\sigma}_{h}^*\|$ \\
\hline

             512    &2.767e-1	&1.717e-1	&1.984e-1&1.572e-1&3.317e-2&1.311e-1\\
             2048      &1.225e-1	&5.578e-2	&6.133e-2&6.094e-2&8.866e-3&3.765e-2\\
             8192       &5.902e-2 	 &2.233e-2 	 &2.326e-2&2.716e-2&2.410e-3&1.063e-2\\
             32768      &2.920e-2 	 &1.030e-2 	 &1.043e-2&1.303e-2&6.509e-4&2.960e-3\\
\hline
order &1.079&1.350&1.415&1.194&1.889&1.823\\
\hline
\end{tabular}
\label{NStab}
\end{table}

\section*{Appendix}
\begin{figure}[tbhp]
\centering
\includegraphics[width=12cm,height=5.0cm]{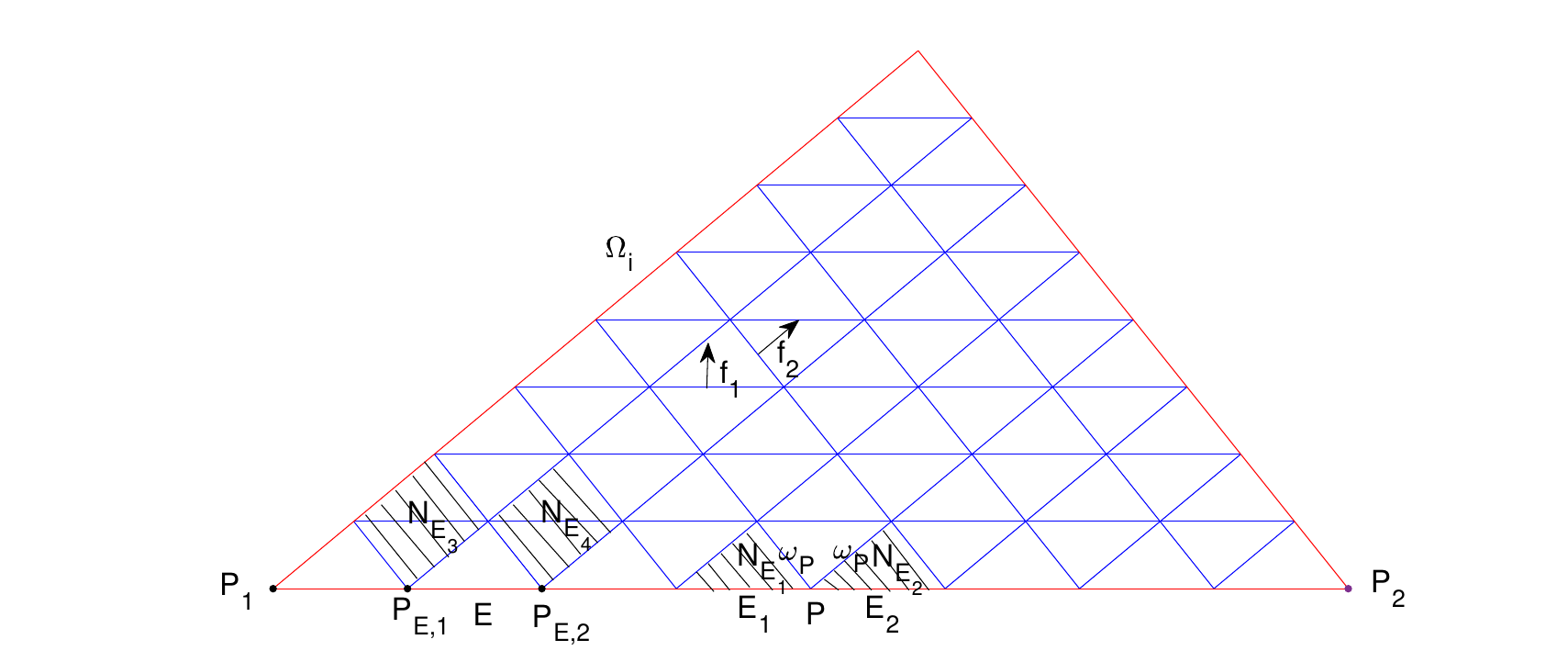}
\caption{A uniform grid of $\Omega_i.$}
\label{uniform}
\end{figure}


\begin{proof}[Proof of Lemma \ref{varerror}]
Let $\{\Omega_{i}\}_{i=1}^N$ be a polygonal partition of $\Omega$ and $\Th|_{\Omega_i}$ be a uniform grid for each $i$. Let $\bm{f}_1$, $\bm{f}_2$ be fixed unit normals to two edges of an arbitrary but \emph{fixed} triangle in $\Th|_{\Omega_i}$.  Let $\bm{S}=\begin{pmatrix}\bm{f}_1^\top\\\bm{f}_2^\top\end{pmatrix}$ and recall $\bm{e}_1=(1,0)$, $\bm{e}_2=(0,1).$  We have 
\begin{equation}\label{totalvar}
\begin{aligned}
&\int_{\Omega_i}[\bm{\bm{D}}(\bm{v}-\widetilde{\Pi}_h\bm{v})]\cdot\nabla^\perp w_hdx=\int_{\Omega_i}\bm{S}^{-\top}[\bm{\bm{D}}(\bm{v}-\widetilde{\Pi}_h\bm{v})]\cdot\bm{S}\nabla^\perp w_hdx\\
&\qquad=\sum_{j=1}^2\int_{\Omega_i}[\bm{e}_j\bm{S}^{-\top}\bm{\bm{D}}(\bm{v}-\widetilde{\Pi}_h\bm{v})](\bm{f}_j^\top\nabla^\perp w_h)dx:=\sum_{j=1}^2I_j.
\end{aligned}
\end{equation}
For $j=1,2$, let $\mathcal{E}^o_j$ and $\mathcal{E}^\partial_j$ be the set of interior (inside $\Omega_i$) and boundary (on $\partial\Omega_i$) edges orthogonal to $\bm{f}_j$, respectively. Let $N_{E}$ denote the region which is the union of triangles sharing $E$ as an edge. To estimate $I_1$, let $\Omega_i$ be partitioned into parallelograms $N_E$ with $E\in\mathcal{E}_1^o$ and boundary triangles $N_E$ with $E\in\mathcal{E}^\partial_1$, see Figure \ref{uniform}, where $E_1, E_2\in\mathcal{E}_1^\partial$, $E_3, E_4\in\mathcal{E}_1^o$. For any $E\in\mathcal{E}_1^o$, note that $\bm{f}_1^\top\nabla^\perp w_h$ is single-valued across $E$ and thus constant on $N_E.$ For $E\in\mathcal{E}_1^\partial$, let $P_{E,1}$ and $P_{E,2}$ be the two endpoints of $E$ and $h_E$ denote the length of $E.$ Then
\begin{equation}
\begin{aligned}
I_1&=\sum_{E\in\mathcal{E}^o_1\cup\mathcal{E}^\partial_1}\int_{N_E}[\bm{e}_1\bm{S}^{-\top}\bm{\bm{D}}(\bm{v}-\widetilde{\Pi}_h\bm{v})](\bm{f}_1^\top\nabla^\perp w_h)dx\\
&=\sum_{E\in\mathcal{E}^o_1}(\bm{f}_1^\top\nabla^\perp w_h)\bm{e}_1\bm{S}^{-\top}\bm{\bm{D}}\int_{N_E}(\bm{v}-\widetilde{\Pi}_h\bm{v})dx\\
&+\sum_{E\in\mathcal{E}^\partial_1}h_E^{-1}\bm{e}_1\bm{S}^{-\top}\bm{\bm{D}}\int_{N_E}(\bm{v}-\widetilde{\Pi}_h\bm{v})dx\big(w_h(P_{E,2})-w_h(P_{E,1})\big)\\
&:=I_{11}+I_{12}.
\end{aligned}
\end{equation}
When $N_E$ is an exact parallelogram, it is shown in Equation (3.15) of \cite{Brandts1994} that
\begin{equation}\label{cancel1}
\left|\int_{N_E}(\bm{v}-\widetilde{\Pi}_h\bm{v})dx\right|\lesssim h^3|\bm{v}|_{H^2(N_E)}.
\end{equation}
Using \eqref{cancel1} and the Cauchy-Schwarz inequality, we have
\begin{equation}\label{I11}
\begin{aligned}
I_{11}&\lesssim\sum_{E\in\mathcal{E}^o_1}h^{-1}\|\nabla^\perp w_h\|_{L^2(N_E)}\left|\int_{N_E}(\bm{v}-\widetilde{\Pi}_h\bm{v})dx\right|\\
&\lesssim\sum_{E\in\mathcal{E}^o_1}h^2\|\nabla^\perp w_h\|_{L^2(N_E)}|\bm{v}|_{H^2(N_E)}\lesssim h^2\|\nabla^\perp w_h\| |\bm{v}|_2.
\end{aligned}
\end{equation}
Let $\mathcal{N}$ denote the collection of endpoints of edges in $\mathcal{E}_1^\partial$,  
$$\mathcal{N}_1:=\{P\in\mathcal{N}: \text{ $P$ is not a corner of $\partial\Omega_i$}\},$$ and $\mathcal{N}_2:=\mathcal{N}\backslash\mathcal{N}_1$. For instance, $\mathcal{N}_2=\{P_1, P_2\}$ in Figure \ref{uniform}. For $P\in\mathcal{N}_1,$ let $E_1$, $E_2$ be the two boundary edges sharing $P$ and $\omega_P$ be the union of three triangles having $P$ as a vertex. For $P\in\mathcal{N}_2,$ let $E$ denote the unique edge in $\mathcal{E}_1^\partial$ having $P$ as a vertex.  Rearranging the summation in $I_{12}$, we have
\begin{equation}\label{I12inter}
\begin{aligned}
I_{12}&=\sum_{P\in\mathcal{N}_1}h_{E_1}^{-1}\bm{e}_1\bm{S}^{-\top}\bm{\bm{D}}\left(\int_{N_{E_1}}(\bm{v}-\widetilde{\Pi}_h\bm{v})dx-\int_{N_{E_2}}(\bm{v}-\widetilde{\Pi}_h\bm{v})dx\right) w_h(P)\\
&+\sum_{P\in\mathcal{N}_2}h_E^{-1}\bm{e}_1\bm{S}^{-\top}\bm{\bm{D}}\int_{N_{E}}(\bm{v}-\widetilde{\Pi}_h\bm{v})dx w_h(P).
\end{aligned}
\end{equation}
For $P\in\mathcal{N}_1$, Equation (3.15) in \cite{HuMaMa2021} shows that
\begin{equation}\label{cancel2}
\left|\int_{N_{E_1}}(\bm{v}-\widetilde{\Pi}_h\bm{v})dx-\int_{N_{E_2}}(\bm{v}-\widetilde{\Pi}_h\bm{v})dx\right|\lesssim h^3|\bm{v}|_{H^2(\omega_P)}.
\end{equation}
Without loss of generality, let $\int_\Omega w_hdx=0.$ Then the discrete Sobolev and Poincar\'e inequalities yield
\begin{equation}\label{disSob}
\|w_h\|_{L^\infty(\Omega)}\lesssim|\log h|^\frac{1}{2}\|w_h\|_1\lesssim|\log h|^\frac{1}{2}\|\nabla^\perp w_h\|.
\end{equation}
Standard interpolation error estimate yields
\begin{equation}\label{stand}
\left|\int_{N_{E}}(\bm{v}-\widetilde{\Pi}_h\bm{v})dx\right|\lesssim h^2|\bm{v}|_{H^1(N_E)}.
\end{equation}
It then follows from \eqref{I12inter}--\eqref{stand}, and the Cauchy--Schwarz inequality that
\begin{equation}\label{I12}
\begin{aligned}
I_{12}&\lesssim\big(\sum_{P\in\mathcal{N}_1}h^2|\bm{v}|_{H^2(\omega_P)}+\sum_{P\in\mathcal{N}_2}h|\bm{v}|_{H^1(N_E)}\big)\|w_h\|_{L^\infty(\Omega)}\\
&\lesssim\big(h^2|\bm{v}|_{W^2_\infty(\Omega)}\sum_{P\in\mathcal{N}_1}|\omega_P|^\frac{1}{2}+h|\bm{v}|_{W^1_\infty(\Omega)}\sum_{P\in\mathcal{N}_2}|N_E|^\frac{1}{2}\big)|\log h|^\frac{1}{2}\|\nabla^\perp w_h\|\\
&\lesssim h^2\|\bm{v}\|_{W^2_\infty(\Omega)}|\log h|^\frac{1}{2}\|\nabla^\perp w_h\|.
\end{aligned}
\end{equation}
We finally conclude the proof from \eqref{totalvar}, \eqref{I11}, \eqref{I12}, the same analysis for $I_2$ and all pieces in the partition $\{\Omega_i\}_{i=1}^N$.
\qed\end{proof}

\section*{Declarations}
\vspace{0.2cm}
\noindent\textbf{Funding} The authors did not receive support from any organization for this work.

\vspace{0.2cm}
\noindent\textbf{Data Availability} Data sharing is not applicable to this article as no datasets were generated or analysed during
the current study.

\vspace{0.2cm}
\noindent\textbf{Conflicts of interest} The authors have no relevant financial or non-financial interests to disclose.

\vspace{0.2cm}
\noindent\textbf{Code availability} The code used in this study is available from the authors upon request.

\bibliographystyle{spmpsci}

\end{document}